\newif\ifarxiv\arxivfalse
\arxivtrue

 \ifarxiv
     \documentclass{article}
 \else
    \documentclass{svjour3}
    \smartqed 
 \fi

\usepackage[utf8]{inputenc}
\usepackage{amsmath,amssymb,amsthm,amscd,graphicx,enumerate}
\numberwithin{equation}{section}
\usepackage{geometry}
\usepackage{authblk}
\usepackage{xargs}
\usepackage[linesnumbered]{algorithm2e}
\usepackage[most]{tcolorbox}
\usepackage[caption=false]{subfig}
\usepackage{float}
\usepackage{mathtools}

\usepackage{array,tabularx,colortbl,makecell}
\usepackage{tcolorbox}

\DeclarePairedDelimiter\floor{\lfloor}{\rfloor}

\tcbset{colback=yellow!10!white, colframe=black, 
        highlight math style= {enhanced, 
            colframe=red,colback=red!10!white,boxsep=0pt}
       }
 
 \newcommand\innprod[2]{\left\langle{}#1{},{}#2{}\right\rangle}%
 \newcommand\func[3]{#1:#2\rightarrow#3}
 \newcommand\E{\mathbb E}%
 \newcommand\R{\mathbb R}%
 \newcommand\eL{\mathcal{L}}%
 \newcommand\A{\mathcal{A}_H^{p}}%
 \DeclareMathOperator*{\argmin}{argmin}
 \newcommand\PS{\psi_{\bar x, H}^p}%
 \newcommand\br{\beta_\rho}
 \newcommand\prox{\mathrm{prox}_{F/H}^p}%
 \newcommand\sprox{\mathrm{sprox}_{F/H}^p}%
 \newcommand\dom{\mathrm{dom}}%
 \newcommand\set[1]{\left\{{}#1{}\right\}} 
 \newcommandx\seq[3][{2=k\geq 0},{3={}}]{\{#1\}_{#2}^{#3}}

 \newtheorem{thm}{Theorem}
 \newtheorem{lem}[thm]{Lemma}
 \newtheorem{prop}[thm]{Proposition}
 
 \newtheorem{rem}[thm]{Remark}
 \newtheorem{exa}[thm]{Example}
 \newtheorem{defin}[thm]{Definition}
 
 \makeatletter
    \@namedef{subjclassname@2020}{%
      \textup{2020} Mathematics Subject Classification}
 \makeatother

 \ifarxiv
    \title{{\bf High-order methods beyond the classical complexity bounds, II: inexact high-order proximal-point methods with segment search}}
    \author{
    {Masoud Ahookhosh\footnote{Department of Mathematics, University of Antwerp, Middelheimlaan 1, B-2020 Antwerp, Belgium.
			E-mail: masoud.ahookhosh@uantwerp.be}} \qquad  {Yurii Nesterov\footnote{Center for Operations Research and Econometrics (CORE) and Department of Mathematical Engineering (INMA), Catholic University of Louvain (UCL), 34voie du Roman Pays, 1348 Louvain-la-Neuve, Belgium. E-mail: Yurii.Nesterov@uclouvain.be\\ This project has received funding from the European Research Council (ERC) under the European Union's Horizon 2020 research and innovation programme (grant agreement No. 788368).}}
    }
    \providecommand{\keywords}[1]{\textbf{Keywords:} #1}
    
\begin{document}
    \maketitle
     \begin{abstract}
        A bi-level optimization framework (BiOPT) was proposed in \cite{ahookhosh2020inexact} for convex composite optimization, which is a generalization of bi-level unconstrained minimization framework (BLUM) given in \cite{nesterov2020inexact}. 
        In this continuation paper, we introduce a $p$th-order proximal-point segment search operator which is used to develop novel accelerated methods. Our first algorithm combines the exact element of this operator with the estimating sequence technique to derive new iteration points, which shown to attain the convergence rate $\mathcal{O}(k^{-(3p+1)/2})$, for the iteration counter $k$.
        We next consider inexact elements of the high-order proximal-point segment search operator to be employed in the BiOPT framework. We particularly apply the accelerated high-order proximal-point method at the upper level, while we find approximate solutions of the proximal-point segment search auxiliary problem by a combination of non-Euclidean composite gradient and bisection methods. For $q=\floor{p/2}$, this amounts to a $2q$th-order method with the convergence rate $\mathcal{O}(k^{-(6q+1)/2})$ (the same as the optimal bound of $2q$th-order methods) for even $p$ ($p=2q$) and the superfast convergence rate $\mathcal{O}(k^{-(3q+1)})$ for odd $p$ ($p=2q+1$).
     \end{abstract}
    

\else
 \begin{document}
 \journalname{Mathematical Programming}
\title{High-order methods beyond the classical complexity bounds, II: inexact high-order proximal-point methods with segment search}
\titlerunning{High-order methods beyond the classical complexity bounds, II}
\author{%
		Masoud Ahookhosh \and
		Yurii Nesterov%
	}
	\institute{
	M. Ahookhosh
		\at
			Department of Mathematics, University of Antwerp, Middelheimlaan 1, B-2020 Antwerp, Belgium.\\
			\email{masoud.ahookhosh@uantwerp.be}%
	\and 
	Y. Nesterov 
	    \at
              Center for Operations Research and Econometrics (CORE) and Department of Mathematical Engineering (INMA), Catholic University of Louvain (UCL), 34voie du Roman Pays, 1348 Louvain-la-Neuve, Belgium. \\
              \email{Yurii.Nesterov@uclouvain.be} \\          %
              This project has received funding from the European Research Council (ERC) under the European Union's Horizon 2020 research and innovation programme (grant agreement No. 788368).
    }
    \maketitle

\begin{abstract}
    A bi-level optimization framework (BiOPT) was proposed in \cite{ahookhosh2020inexact} for convex composite optimization, which is a generalization of bi-level unconstrained minimization framework (BLUM) given in \cite{nesterov2020inexact}. 
    In this continuation paper, we introduce a $p$th-order proximal-point segment search operator which is used to develop novel accelerated methods. Our first algorithm combines the exact element of this operator with the estimating sequence technique to derive new iteration points, which shown to attain the convergence rate $\mathcal{O}(k^{-(3p+1)/2})$, for the iteration counter $k$.
    We next consider inexact elements of the high-order proximal-point segment search operator to be employed in the BiOPT framework. We particularly apply the accelerated high-order proximal-point method at the upper level, while we find approximate solutions of the proximal-point segment search auxiliary problem by a combination of non-Euclidean composite gradient and bisection methods. For $q=\floor{p/2}$, this amounts to a $2q$th-order method with the convergence rate $\mathcal{O}(k^{-(6q+1)/2})$ (the same as the optimal bound of $2q$th-order methods) for even $p$ ($p=2q$) and the superfast convergence rate $\mathcal{O}(k^{-(3q+1)})$ for odd $p$ ($p=2q+1$).
\end{abstract}

\fi

\section{Introduction}\label{sec:intro}

\paragraph{{\bf Motivation.}}
Complexity theory plays a crucial role in analysis and assessment of the efficiency of convex optimization algorithms and aims at understanding their worst-case behavior. In a common paradigm, there is a one-to-one correspondence between methodologies and problem classes, which roughly speaking means that for a fixed class of problems (with some fixed set of properties) and a specified method, there is a complexity bound that cannot be improved. This consequently led to the definition of {\it optimal complexity} and {\it optimal methods} in the field of convex optimization. As such, during the last decades, this topic has received much attention in the optimization community. For instance, if a problem is $p$-times differentiable with Lipschitz (H\"older) continuous $p$th derivatives, then the best complexity for $p$th-order methods is of order $\mathcal{O}(\varepsilon^{-2/(3p+1)})$ for the accuracy parameter $\varepsilon>0$; see, e.g., \cite{ahookhosh2019accelerated,ahookhosh2017optimal,ahookhosh2018solving,nesterov2005smooth,nesterov2013gradient,nesterov2015universal,neumaier2016osga} for $p=1$ and \cite{agarwal2018lower,arjevani2019oracle,birgin2017worst,gasnikov2019optimal,jiang2019optimal} for an arbitrary $p$.

Recently, the classical optimal complexity paradigm has been challenged in several angles. In the first attempt, a {\it second-order method} with the convergence rate $\mathcal{O}(k^{-4})$ has been proposed in \cite{nesterov2020superfast}, which is faster than the classical lower bound $\mathcal{O}(k^{-7/2})$ for second-order methods. In this method, a third-order tensor method was implemented, where its auxiliary problem is handled by a Bregman gradient method requiring second-order oracles; however, the Lipschitz continuity of third derivatives was assumed while in the classical setting we only require the Lipschitz continuity of Hessian. As such, the convergence rate $\mathcal{O}(k^{-4})$ for this method is not a contradiction with classical complexity theory for second-order methods. 

More recently, a {\it bi-level unconstrained minimization} (BLUM) framework was proposed by Nesterov in \cite{nesterov2020inexact} using the {\it high-order proximal-point operator}
\begin{equation}\label{eq:prox0}
    \mathrm{prox}_{f/H}^p(\bar x)=\argmin_{x\in\E} \set{f(x)+\tfrac{H}{p+1} \|x-\bar x\|^{p+1}},
\end{equation}
for $p\geq 1$, $\bar x \in\E$, and a convex function $\func{f}{\E}{\R}$. This framework involves two levels of methodologies, where the upper-level involves a scheme using the $p$the-order proximal-point operator (for arbitrary $p$), and the lower-level is a scheme for finding an inexact solution of the corresponding proximal-point minimization. Applying the BLUM framework to twice smooth unconstrained problems with $p=3$, using a Nesterov-type accelerated scheme at the upper level, and solving the auxiliary problem by a Bregman gradient method lead to a second-order method with the superfast convergence rate $\mathcal{O}(k^{-4})$. The BLUM framework was further extended in \cite{nesterov2020inexactSS} by replacing \eqref{eq:prox0} with the {\it high-order proximal-point segment search operator}  
\begin{equation*}\label{eq:proxLS0}
    \mathrm{sprox}_{f/H}^p(\bar x, \bar u)=\argmin_{x\in\E,~\tau\in[0,1]} \set{f(x)+ \tfrac{H}{p+1} \|x-\bar x-\tau \bar u\|^{p+1}},
\end{equation*}
for a direction $\bar u\in\E$. Executing the BLUM framework with segment search to twice smooth unconstrained problems with $p=3$, employing a Nesterov-type accelerated method at the upper level and a Bregman gradient method in the lower level, results to a second-order method with the superfast convergence rate $\mathcal{O}(k^{-5})$.

In the predecessor of the current paper (i.e., \cite{ahookhosh2020inexact}), we recently introduced a \textit{Bi-level OPTimization} (BiOPT) framework for convex composite minimization, which is a generalization the BLUM framework. Similarly, the BiOPT framework involves two levels of methodologies. While, at the upper level of BiOPT, the objective function is regularized by the $(p+1)$th-order proximal term resulting in a \textit{$p$th-order composite proximal-point operator}
\begin{equation}\label{eq:prox00}
    \prox(\bar x)=\argmin_{x\in\dom \psi} \set{f(x)+\psi(x)+\tfrac{H}{p+1} \|x-\bar x\|^{p+1}},
\end{equation}
for $H>0$ and $p\geq 1$. For $p=1$, this operator reduces to the classical proximal operator introduced by Martinet in \cite{martinet1970breve,martinet1972determination} that was further studied by Rockafellar \cite{rockafellar1976monotone} where $H$ is replaced by a sequence of positive numbers $\seq{H_k}$.
At the lower level the corresponding proximal-point auxiliary problem is solved approximately either by a single iteration of the $p$th-order tensor method or by a lower-order non-Euclidean composite gradient scheme, for an arbitrary $p$. For $q=\floor{p/2}$, employing the accelerated proximal-point method at the upper level and the non-Euclidean composite gradient scheme at the lower level, the proposed $2q$th-order method attains the convergence rate $\mathcal{O}(k^{-(p+1)})$.  We note that this method for $p=3$ leads to a superfast method, while for the other choices of $p$ it is suboptimal. In the current work, we continue the same stream and introduce some methods attaining the optimal rate for even $p$ and a superfast rate for odd $p$.

\vspace{-4mm}
\subsection{Content}\label{sec:content}
The main goal of this paper is to introduce a high-order proximal-point segment search operator for convex composite minimization and to extend the BiOPT framework by replacing the the high-order proximal-point operator with its segment search counterpart. Our objective function is sum of two convex functions, where both can be nonsmooth. We first  regularize the objective function by a power of the Euclidean norm $\|\cdot\|^{p+1}$ with $p\geq 1$, which includes a segment search along some specific direction. A minimization of this augmented function leads to a high-order proximal-point segment search operator, which will be minimized approximately in a reasonable cost. In Section~\ref{sec:exact}, we design our first algorithm  that is a combination of the high-order proximal-point segment search operator with the estimating sequence technique attaining the convergence rate $\mathcal{O}(k^{-(3p+1)/2})$ (see Theorem~\ref{thm:aihoppaSSConvRate}), for the iteration counter $k$. 

We present our second algorithm in Section~\ref{sec:ihoppSegmSearch},
which is based on inexact solution of high-order proximal-point segment search minimization. The algorithm involves a Nesterov-type acceleration at the upper level and a combination of inexact solution of high-order proximal-point minimization \eqref{eq:prox00} and a bisection scheme (for applying the segment search) in the lower level. We show the convergence rate $\mathcal{O}(k^{-(3p+1)/2})$ (see Theorem~\ref{thm:convRateAlg6}) for this inexact method. We finally derive a bound on the primal-dual gap that can be used as an implementable stopping criterion for the proposed method (see Proposition~\ref{prop:primDualGap}). 

We next extend the bi-level optimization framework (BiOPT) given in \cite{ahookhosh2020inexact} by applying the high-order proximal-point segment search minimization. The BiOPT involves two levels of methodologies. At the upper level, for an arbitrary $p$, we develop an inexact accelerated high-order proximal-point schemes with segment search using the estimating sequence technique. At the lower level, we use a combination of the non-Euclidean composite gradient scheme and the bisection technique to find an element of the high-order proximal-point segment search operator. 
To implement the non-Euclidean composite gradient algorithm, we need a $2q$th-order oracle of the scaling function (with $q=\floor{p/2}$), i.e., the proposed algorithm is a $2q$th-order method. If $p$ is even ($p=2q$), we end up with a $2q$th-order method with the convergence rate $\mathcal{O}(k^{-(6q+1)/2})$, which is the same as the optimal convergence rate of $2q$th-order methods. For odd $p$ ($p=2q+1$), we end up with a $2q$th-order method with the superfast convergence rate $\mathcal{O}(k^{-(3q+1)})$. We note that it is not contradicting the classical complexity theory because the proposed method requires the boundedness of the $(p+1)$th derivative of the smooth part of the objective function, which is stronger than the Lipschitz continuity of $2q$th-order derivative needed in classical setting (see Table~\ref{tab:complexity} for more details). 

We finally give some conclusion in Section~\ref{sec:conclusion}.

\subsection{Notation and generalities} \label{sec:notation}
In what follows, we denote by $\E$ a finite-dimensional real vector space, and by $\E^*$ its dual spaced composed by linear functions on $\E$. For such a function $s\in \E^*$, we denote by $\innprod{s}{x}$ its value at $x\in \E$.

Let us measure distances in $\E$ and $\E^*$ in a Euclidean norm. For that, using a self-adjoint positive-definite operator $\func{B}{\E}{\E^*}$ (notation $B=B^* \succ 0$), we define
\begin{align*}
    \|x\|=\innprod{Bx}{x}^{1/2},\quad x\in \E,\quad \|g\|_{*} = \innprod{g}{B^{-1}g}^{1/2}, \quad g\in \E^*.
\end{align*}
Sometimes, it will be convenient to treat $x\in \E$ as a linear operator from $\R$ to $\E$, and $x^*$ as a linear operator from $\E^*$ to $\R$. In this case, $xx^*$ is a linear operator from $\E^*$ to $\E$, acting as follows:
\begin{align*}
    (xx^*)g = \innprod{g}{x} x\in\E,\quad g\in\E^*.
\end{align*}
For a smooth function $\func{f}{\E}{\R}$, we denote by $\nabla f(x)$ its gradient and by $\nabla^2 f(x)$ its Hessian evaluated at point $x\in \E$. Note that
\begin{align*}
    \nabla f(x)\in\E^*, \quad \nabla^2 f(x)h\in\E^*, \quad x,h\in\E.
\end{align*}
We denote by $\ell_{\bar x}(\cdot)$ the linear model of convex function $f(\cdot)$ at point $\bar x\in\E$ given by
\begin{equation}
    \ell_{\bar x}(x) = f(\bar x)+\innprod{\nabla f(\bar x)}{x-\bar x}, \quad x\in\E.
\end{equation}
Using the above norm, we can define the standard Euclidean prox-functions
\begin{align*}
    d_{p+1}(x)=\tfrac{1}{p+1}\|x\|^{p+1}, \quad x\in\E.
\end{align*}
where $p\geq 1$ is an integer parameter. These functions have the following derivatives:
\begin{equation}\label{eq:dp1Der}
    \begin{array}{rcl}
        \nabla d_{p+1}(x) &= &\|x\|^{p-1}Bx, \quad x\in\E, \\
        \nabla^2 d_{p+1}(x) &= &\|x\|^{p-1}B+(p-1)\|x\|^{p-3}Bxx^*B \succeq \|x\|^{p-1}B. 
    \end{array}
\end{equation}
Note that the function $d_{p+1}(\cdot)$ is $2^{2-p}$-uniformly convex (see, for example, \cite[Lemma 4.2.3]{nesterov2018lectures}):
\begin{equation}\label{eq:dp1LowerBound}
    d_{p+1}(y) \geq d_{p+1}(x)+\innprod{d_{p+1}(x)}{y-x}+\tfrac{2^{2-p}}{p+1} \|y-x\|^{p+1}, \quad x,y\in \E.
\end{equation}

In what follows, we often work with directional derivatives. For $p\geq 1$, denote by 
\begin{align*}
    D^pf(x)[h_1,\ldots,h_p]
\end{align*}
the directional derivative of function $f$ at $x$ along directions $h_i\in\E$, $i= 1,...,p$. Note that $D^pf(x)[\cdot]$ is a asymmetric $p$-linear form. Its norm is defined in a standard way:
\begin{equation}\label{eq:dpfhiNorm}
    \|D^pf(x)\| = \max_{h_1,\ldots,h_p} \set{\left|D^pf(x)[h_1,\ldots,h_p]\right| ~:~ \|h_i\|\leq 1,~ i=1,\ldots,p}.
\end{equation}
If all directions $h_1,\ldots,h_p$ are the same, we apply the notation
\begin{align*}
    D^pf(x)[h]^p,\quad h\in\E.
\end{align*}
Note that, in general, we have (see, for example, \cite[Appendix 1]{nesterov1994interior})
\begin{equation}\label{eq:dpfhpNorm}
    \|D^pf(x)\| = \max_{h} \set{\left|D^pf(x)[h]^p\right| ~:~ \|h\|\leq 1}.
\end{equation}
If the considered function belongs to the problem class $\mathcal{F}_p$, convex and $p$ times continuously differentiable functions on $\E$, then we denote by $M_p(f)$ its uniform upper bound for its $p$th derivative:
\begin{equation}\label{eq:mpf}
    M_p(f)=\sup_{x\in\E} \|D^pf(x)\|.
\end{equation}


\section{Exact high-order proximal-point segment search methods}\label{sec:exact}
Let us consider the convex composite minimization problem
\begin{equation}\label{eq:prob}
    \min_{x\in\dom \psi}~\set{F(x)=f(x)+\psi(x)},
\end{equation}
where $\func{f}{\E}{\R}$ is a closed convex and possibly non-differentiable function, and $\func{\psi}{\E}{\R}$ is a simple subdifferentiable closed convex function on its domain such that $\dom \psi\subseteq \mathrm{int}(\dom f)$. We further assume that \eqref{eq:prob} has at least one optimal solution $x^*\in\dom \psi$ and denote $F^*=F(x^*)$. Note that this class of problems is general enough to encapsulate wide variety problems from many applications fields. As an important example, for the simple closed convex set $Q\subseteq\E$, the simple constrained problem
\begin{align*}
    \begin{array}{ll}
        \min & f(x) \\
        \mathrm{s.t.} & x\in Q
    \end{array}
\end{align*}
can be rewritten as a composite minimization \eqref{eq:prob} with $F(x)=f(x)+\delta_Q(x)$, where $\delta_Q(\cdot)$ is the indicator function of the set $Q$ and $\dom \psi = Q$.
  
Let us define the \textit{$p$th-order composite proximal-point segment search} operator $\func{\sprox}{\dom \psi\times\E}{\E\times\R}$ given by
\begin{equation}\label{eq:proxLS}
    \displaystyle \sprox(\bar x, \bar u)=\argmin_{x\in\dom \psi,~\tau\in[0,1]} \set{\PS(x)=f(x)+\psi(x)+H d_{p+1}(x-\bar x-\tau \bar u)},
\end{equation}
where a point $\bar x\in\E$, a direction $\bar u\in\E$, and a regularization parameter $H>0$ are given. 
Setting $y=x-\tau \bar u$, the problem \eqref{eq:proxLS} can be translated to the minimization problem
\begin{equation}\label{eq:constProx}
    \min_{y\in\E, \tau\in [0,1]} \set{f(y+\tau \bar u)+\psi(y+\tau \bar u)+H d_{p+1}(y-\bar x)},
\end{equation}
where $(x_+,\tau_+)$ is its solution with $x_+=x(\bar x,\bar u)$ and $\tau_+=\tau(\bar x,\bar u)$. Writing the first-order optimality conditions for this constrained problem with respect to $\tau$, there exists a subgradient $g_+\in\partial\psi(x_+)$ such that
\begin{equation}\label{eq:foOptConds}
    (\tau-\tau_+)\innprod{\nabla f(x_+)+g_+}{\bar u}\geq 0 \quad \forall \tau\in [0,1],
\end{equation}
see, e.g., \cite[Theorem~3.1.24]{nesterov2018lectures}. In addition, setting $y_+=\bar x+\tau_+ \bar u$ and $r_+=\|x_+-y_+\|$, the optimality conditions for \eqref{eq:proxLS} with respect to $x$ implies that
\begin{align*}
    \innprod{\nabla f(x_+)+Hr_+^{p-1} B(x_+-y_+)}{x-x_+}+\psi(x) \geq \psi(x_+),
\end{align*}
i.e., $g_+=-\nabla f(x_+)-Hr_+^{p-1} B(x_+-y_+)\in\partial\psi(x_+)$, which eventually leads to
\begin{equation}\label{eq:inprodnablaFY+}
    \innprod{\nabla f(x_+)+g_+}{y_+-x_+} = Hr_+^{p+1}=\left(\tfrac{1}{H}\right)^{\tfrac{1}{p}} \|\nabla f(x_+)+g_+\|_*^{\tfrac{p+1}{p}}.
\end{equation}

\begin{rem}[exact $p$th-order composite proximal-point segment search] \label{rem:exactSol}
	Let us write the optimality conditions for the minimization problem \eqref{eq:proxLS}, i.e.,
	\begin{equation}\label{eq:opcProxLS}
	\left\{
		\begin{array}{l}
		\nabla f(x_+)+\partial \psi(x_+)+H \|x_+-\bar x-\tau_+ \bar u\|^{p-1} B(x_+-\bar x-\tau_+ \bar u) \ni 0,\vspace{2mm}\\
		-H \|x_+-\bar x-\tau_+ \bar u\|^{p-1}\innprod{x_+-\bar x-\tau_+ \bar u}{\bar u}+N_{[0,1]}(\tau_+) \ni 0,
		\end{array}
	\right.
	\end{equation}
	where
	\begin{equation}\label{eq:normalCone01}
		N_{[0,1]}(\tau_+) = \left\{
		\begin{array}{ll}
			0                      &~~ \mathrm{if}~ \tau_+\in (0,1),\\
			\left(-\infty,0\right] &~~ \mathrm{if}~ \tau_+=0,\\
			\left[0,+\infty\right) &~~ \mathrm{if}~ \tau_+=1,\\
			\emptyset              &~~ \mathrm{otherwise}.
		\end{array}
		\right.
	\end{equation}
	It follows from the first inclusion in \eqref{eq:opcProxLS} that there exists a subgradient $g_+\in\partial \psi(x_+)$ such that
	\begin{equation}\label{eq:optCondsEq}
		\nabla f(x_+)+g_++H \|x_+-\bar x-\tau_+ \bar u\|^{p-1} B(x_+-\bar x-\tau_+ \bar u) = 0,
	\end{equation}
	which clearly implies that there exists a constant $\theta$ such that $B(x_+-\bar x-\tau_+ \bar u)1=\theta (\nabla f(x_+)+g_+)$. Substituting this in the first inclusion of \eqref{eq:opcProxLS} results to
	\begin{align*}
		\theta = H^{-\tfrac{1}{p}} \|\nabla f(x_+)+g_+\|^{\tfrac{1-p}{p}}.
	\end{align*}
	We now consider three cases: (i) $\tau_+\in (0,1)$; (ii) $\tau_+=0$; and (iii) $\tau_+=1$. In Case (i), it follows from the second inclusion in \eqref{eq:opcProxLS} that
	\begin{align*}
		\tau_+= \innprod{x-\bar x}{\tfrac{\bar u}{\|\bar u\|^2}}.
	\end{align*}
	In Cases (ii) and (iii), this inclusion leads to the inequalities $\innprod{x-\bar x}{\bar u}\geq 0$ and $\innprod{x-\bar x-\bar u}{\bar u}\leq 0$, respectively. Accumulating these three conditions $(x_+,\tau_+)$ can be computed from following relations
	\begin{equation}
		x_+ = \left\{
		\begin{array}{ll}
			\bar x + \theta \left(I-\tfrac{\bar u \bar u^T}{\|\bar u\|^2}\right)^{-1} (\nabla f(x_+)+g_+) &~~ \mathrm{if}~ \innprod{x_+-\bar x}{\bar u/\|\bar u\|^2}\in (0,1),\vspace{1mm}\\
			\bar x + \theta (\nabla f(x_+)+g_+) &~~ \mathrm{if}~ \innprod{x_+-\bar x}{\bar u}\geq 0,\vspace{2mm}\\
			\bar x + \bar u + \theta (\nabla f(x_+)+g_+) &~~ \mathrm{if}~ \innprod{x_+-\bar x}{\bar u}\leq 0,
		\end{array}
		\right.
	\end{equation}
	where efficient computation of $(x_+,\tau_+)$ are highly connected to the structure of functions $f(\cdot)$ and $\psi(\cdot)$.
\end{rem}

\begin{exa}[computing $\sprox(\bar x, \bar u)$ for an one-dimensional problem] \label{exa:exactSol}
Let us consider the minimization of the one-dimensional function $\func{F}{\R}{\R}$ given by $F(x)=\tfrac{1}{2}x^2+|x|$, where $x^*=0$ is its unique solution. In the setting of the problem \eqref{eq:prob}, we have $f(x)=\tfrac{1}{2}x^2$ and $\psi(x)=|x|$. We set $p=3$ and $H=1$, and verify an exact solution for the third-order composite proximal-point segment search minimization. Invoking the normal cone \eqref{eq:normalCone01}, we consider three cases: (i) $\tau_+\in (0,1)$; (ii) $\tau_+=0$; and (iii) $\tau_+=1$. In Case (i), we have $\tau_+= (x_+-\bar x)/\bar u$. If $x_+>0$, then $g_+=1$ and from \eqref{eq:opcProxLS} we obtain
\begin{align*}
	x_++1+|x_+-\bar x-\tau_+\bar u|^2(x_+-\bar x-\tau_+\bar u)=0
\end{align*}
leading to $x_+=-1$ that contradicts to $x_+>0$. If $x_+<0$, then $g_+=-1$ and it follows from \eqref{eq:opcProxLS} that
\begin{align*}
	x_+-1+|x_+-\bar x-\tau_+\bar u|^2(x_+-\bar x-\tau_+\bar u)=0
\end{align*}
implying $x_+=1$ that contradicts to $x_+<0$. In case $x_+=0$, if $\tau_+=-\bar x/\bar u\in (0,1)$, then $g_+=0$. Let us now consider Case~(ii) ($\tau_+=0$), where it can be deduced from \eqref{eq:opcProxLS} that 
\begin{align*}
	g_+=1, \quad x_++1+|x_+-\bar x|^2(x_+-\bar x)=0
\end{align*}
if $x_+>0$. This equation has a solution if $\bar x>-1$ (see Subfigure (a) of Figure~\ref{fig:fig1}). For $x_+<0$, we have 
\begin{align*}
	g_+=-1, \quad x_+-1+|x_+-\bar x|^2(x_+-\bar x)=0.
\end{align*}
This equation has a solution if $\bar x<1$ (see Subfigure (b) of Figure~\ref{fig:fig1}). If $x_+=0$, then $g_+=|\bar x|\bar x$. Similarly, in Case~(iii), $\tau_+=1$, it follows from \eqref{eq:opcProxLS} that
\begin{align*}
	g_+=1, \quad x_++1+|x_+-\bar x-\bar u|^2(x_+-\bar x-\bar u)=0
\end{align*}
if $x_+>0$. This equation has a solution if $\bar x+\bar u>-1$. For $x_+<0$, we have 
\begin{align*}
	g_+=-1, \quad x_+-1+|x_+-\bar x-\bar u|^2(x_+-\bar x-\bar u)=0.
\end{align*}
This equation has a solution if $\bar x+\bar u<1$.If $x_+=0$, then $g_+=|\bar x+\bar u|(\bar x+\bar u)$. Summarizing above discussion, we come to
\begin{equation*}
	(x_+,\tau_+) = \left\{
	\begin{array}{ll}
		x_+=0,~ \tau_+=-\bar x/\bar u &~~ \mathrm{if}~ -\bar x/\bar u\in (0,1),\\
		x_+~\mathrm{is~ the~ larger~ root~ of}~ |x-\bar x|^3=|x+1|,~ \tau_+=0 &~~ \mathrm{if}~ \bar u\geq 0~ \& ~\bar x >-1,\\
		x_+~\mathrm{is~ the~ smaller~ root~ of}~ |x-\bar x|^3=|x-1|,~ \tau_+=0 &~~ \mathrm{if}~ \bar u\leq 0~ \& ~\bar x <1,\\
		x_+~\mathrm{is~ the~ larger~ root~ of}~ |x-\bar x-\bar u|^3=|x+1|,~ \tau_+=1 &~~ \mathrm{if}~ \bar u\geq 0~ \& ~\bar x+\bar u >-1,\\
		x_+~\mathrm{is~ the~ smaller~ root~ of}~ |x-\bar x-\bar u|^3=|x-1|,~ \tau_+=1 &~~ \mathrm{if}~ \bar u\leq 0~ \& ~\bar x+\bar u <1,
	\end{array}
	\right.
\end{equation*}
which implies that finding exact solution of $p$th-order proximal-point segment search operator can be complicated or costly even in this one-dimensional problem.
\vspace{-5mm}
    \begin{figure}[H]
        \subfloat[\vspace{-2mm} Solution for $\bar x=-0.5$ and $x_+>0$.]{\includegraphics[width= 7.4cm]{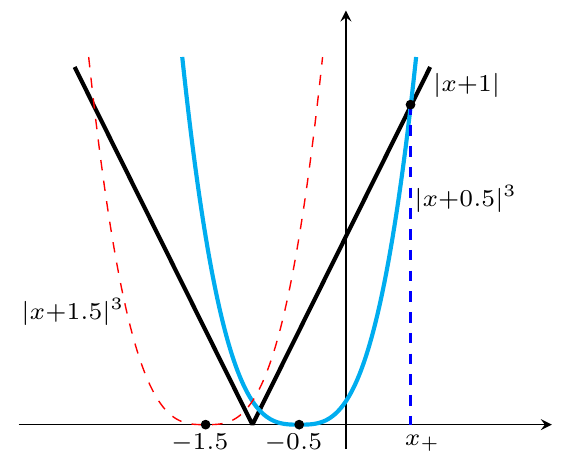}}\qquad
        \subfloat[\vspace{-2mm} Solution for $\bar x=0.5$ and $x_+<0
        455$.]{\includegraphics[width= 7.4cm]{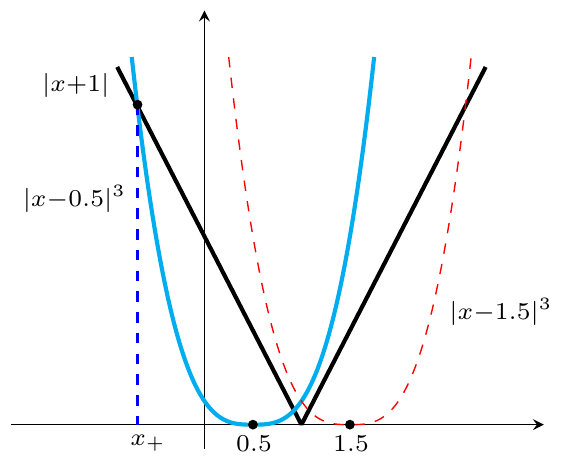}}\\
        \caption{
        Subfigure (a) shows that for $\bar x >-1$ there is a positive solution, while for $\bar x \leq -1$ there is no positive solution, and Subfigure (b) illustrates that for $\bar x <1$ there is a negative solution, while for $\bar x \geq 1$ there is no negative solution.\label{fig:fig1}}
    \end{figure}
\end{exa}

 We now aim at developing our first algorithm based on a combination of the $p$th-order composite proximal-point segment search operator \eqref{eq:proxLS} and an {\it estimating sequence} technique; see, e.g., \cite{nesterov2018lectures}.
To this end, let $\seq{A_k}$ be a sequence of positive numbers generated by $A_{k+1}=A_k+a_{k+1}$ for $a_k>0$. The idea of estimating sequences techniques is to generate a sequence of estimation functions $\seq{\Psi_k(x)}$ of $F$ in such a way that, at each iteration $k\geq 0$, the inequality
\begin{equation}\label{eq:estSeqIneq}
    A_k F(x_k)\leq \Psi_k^*\equiv \min_{x\in\dom\psi} \Psi_k(x), \quad k\geq 0.
\end{equation}
Following \cite{nesterov2020inexact,nesterov2020superfast}, let us set
\begin{equation}\label{eq:Ak}
    A_{k}= \left(\tfrac{c_p}{2}\right)^p \left(\tfrac{k}{p+1}\right)^{p+1}, \quad a_{k+1}=A_{k+1}-A_k, \quad k\geq 0,
\end{equation}
for $c_p=\left(\tfrac{1-\beta}{H}\right)^{1/p}$. Moreover, let $x_0, y_k\in\E$ and $(x_{k+1},\tau_k)$ be a solution of \eqref{eq:proxLS}, and let us define the \textit{estimating sequence} 
\begin{equation}\label{eq:estSeq}
    \Psi_{k+1}(x)=\left\{
    \begin{array}{ll}
        \tfrac{1}{2} \|x-x_0\|^2 &~~\mathrm{if}~ k=0,\vspace{2mm}\\
        \Psi_k(x)+a_{k+1}[\ell_{x_{k+1}}(x)+\psi(x)] &~~\mathrm{if}~ k\geq 1.  
    \end{array}
    \right.
\end{equation}
We here assume that the minimization problem \eqref{eq:proxLS} is solved exactly. A combination of such exact solution with the estimating sequence \eqref{eq:estSeq} leads to Algorithm~\ref{alg:aihoppaSS}.

\RestyleAlgo{boxruled}
\begin{algorithm}[ht!]
\DontPrintSemicolon \KwIn{$x_{0}\in\dom\psi$,~ $H>0$,~ $\upsilon_0=x_0$,~ $A_0=0$,~ $\Psi_0=\tfrac{1}{2}\|x-x_0\|^2$,~ $k=0$;} 
\Begin{ 
    \While {stopping criterion does not hold}{ 
        Set $u_k=\upsilon_k-x_k$;\;
        Compute $(x_{k+1},\tau_k, g)=\sprox(x_k,u_k)$;\;
        Set $g_k=\|\nabla f(x_{k+1})+g\|_*$ for some $g=-\nabla f(x_{k+1})-H\|x_{k+1}-(x_k+\tau_ku_k)\|^{p-1}B(x_{k+1}-(x_k+\tau_ku_k))\in \partial\psi(x_{k+1})$;\;
        Compute $a_{k+1}$ by solving the equation $\tfrac{a_{k+1}^2}{A_{k+1}+a_{k+1}}=\left(\tfrac{1}{H}\right)^{1/p}g_k^{(1-p)/p}$;\;
        Update $\Psi_{k+1}(x)$ by \eqref{eq:estSeq} and set $A_{k+1}=A_k+a_{k+1}$;\;
        Compute $\upsilon_{k+1}=\argmin_{x\in\dom\psi} \Psi_{k+1}(x)$;\;
     } 
}
\caption{Exact High-Order Proximal-Point Segment Search Algorithm\label{alg:aihoppaSS}}
\end{algorithm}

The following two lemmas are essential for giving the convergence rate of the sequence $\seq{x_k}$ generated by Algorithm \ref{alg:aihoppaSS}.

\begin{lem}\cite[Lemma 9]{nesterov2020inexactSS}
\label{lem:ineqPositNum}
    Let the positive numbers $\{\xi_i\}_{i=1}^N$ satisfy the inequality
    \begin{equation}\label{eq:ineqPositNum1}
        \sum_{i=1}^N \omega_i \xi_i^\gamma \leq B,
    \end{equation}
    where $\gamma\geq 1$ and all weights $\omega_i$ are positive. Then,
    \begin{equation}\label{eq:ineqPositNum2}
        \sum_{i=1}^N \tfrac{1}{\xi_i}\geq \tfrac{1}{B^{1/\gamma}} \left(\sum_{i=1}^N\omega_i^{\tfrac{1}{1+\gamma}}\right)^{\tfrac{1+\gamma}{\gamma}}.
    \end{equation}
\end{lem}

\begin{lem}\cite[Lemma 10]{nesterov2020inexactSS}
\label{lem:ineqPositSeq}
    Let the sequence of positive numbers $\{\eta_k\}_{k\geq 0}$ satisfy the inequality
    \begin{equation}\label{eq:ineqPositSeq1}
        \eta_k\geq \theta\left(\sum_{i=1}^N \eta_k^\alpha\right)^\beta,\quad k\geq 1,
    \end{equation}
    where $\theta>0$, $\alpha,\beta\geq 0$, and $\alpha\beta\leq 1$. Then,
    \begin{equation}\label{eq:ineqPositSeq2}
        \eta_k \geq \theta^{\tfrac{1}{1-\alpha\beta}}\left(\alpha\beta+(1-\alpha\beta)k\right)^{\tfrac{\beta}{1-\alpha\beta}},\quad k\geq 1.
    \end{equation}
\end{lem}

Next, our main result of this section will provide the convergence rate of the exact high-order proximal-point segment search algorithm (e.g., Algorithm~\ref{alg:aihoppaSS}) is of order $\mathcal{O}\left(k^{-(3p+1)/2}\right)$, for the iteration counter $k$.

\begin{thm}[convergence rate of Algorithm \ref{alg:aihoppaSS}]\label{thm:aihoppaSSConvRate}
    Let the sequence $\seq{x_k}$ be generated by Algorithm \ref{alg:aihoppaSS}. Then, for any $k\geq 1$, we have
    \begin{align}
        &\Psi_k(x)\leq A_k F(x)+ \tfrac{1}{2} \|x-x_0\|^2, \label{eq:akBkPsisIneq1}\\
        &A_k F(x_k)+B_k\leq \Psi_k^*, \label{eq:akBkPsisIneq2}
    \end{align}
    with $\Psi_k^*=\min_{x\in\dom\psi} \Psi_k(x)$,
    \begin{equation}\label{eq:Bk1}
        B_k=\tfrac{1}{2} \left(\tfrac{1}{H}\right)^{1/p} \sum_{i=0}^{k-1} A_{i+1}g_i^{(p+1)/p}, \quad g_i=\|\nabla f(x_{i+1})+g\|_*,
    \end{equation}
    and
    \begin{align*}
        g=-\nabla f(x_{k+1})-H\|x_{k+1}-(x_k+\tau_ku_k)\|^{p-1}B(x_{k+1}-(x_k+\tau_ku_k))\in \partial\psi(x_{k+1}).
    \end{align*}
    Moreover, for $R_0=\|x_0-x^*\|$, we get
    \begin{equation}\label{eq:akBkIneq}
        F(x_k)-F(x^*) \leq 2^p H R_0^{p+1} \left(1+\tfrac{2(k-1)}{p+1}\right)^{-(3p+1)/2}.
    \end{equation}
\end{thm}

\begin{proof}
    We show \eqref{eq:akBkPsisIneq1} and \eqref{eq:akBkPsisIneq2} by induction. For $k=0$, $A_0=B_0=\Psi_0^*=0$, and then \eqref{eq:akBkPsisIneq1} and \eqref{eq:akBkPsisIneq2} hold. We now assume that this inequality holds for $k$ and show it for $k+1$. By the definition of $\Psi_{k+1}(\cdot)$ and the subgradient inequality, we get
    \begin{align*}
        \Psi_{k+1}(x)&=\Psi_k(x)+a_{k+1} \ell_k(x)
        \leq A_k F(x)+ \tfrac{1}{2} \|x-x_0\|^2+a_{k+1} \ell_k(x)\\
        &\leq A_{k+1} F(x)+ \tfrac{1}{2} \|x-x_0\|^2,
    \end{align*}
    which is the inequality \eqref{eq:akBkPsisIneq1}. Since $\Psi_k(\cdot)$ is $1$-strongly convex, it follows from \eqref{eq:akBkPsisIneq2} and the subgradient inequality that
    \begin{align*}
        \Psi_{k+1}(x) &= \Psi_k(x)+a_{k+1}\ell_k(x)
        \geq \Psi_k(\upsilon_k)+\innprod{\nabla \Psi_k(\upsilon_k)}{x-\upsilon_k}+\tfrac{1}{2} \|x-\upsilon_k\|^2+a_{k+1}\ell_k(x)\\
        &\geq A_k F(x_k)+B_k +\tfrac{1}{2}\|x-\upsilon_k\|^2 +a_{k+1}\ell_k(x)\\
        &\geq A_k F(x_k)+B_k + \min_{x\in\E} \left\{a_{k+1}[f(x_{k+1})+\innprod{\nabla f(x_{k+1}}{x-x_{k+1}}+\psi(x)] +\tfrac{1}{2}\|x-\upsilon_k\|^2\right\}\\
        &\geq A_k F(x_k)+B_k + \min_{x\in\E} \left\{a_{k+1}[F(x_{k+1})+\innprod{\nabla f(x_{k+1})+g}{x-x_{k+1}}] +\tfrac{1}{2}\|x-\upsilon_k\|^2\right\}\\
        &\geq A_k F(x_k)+B_k + a_{k+1}[F(x_{k+1})+\innprod{\nabla f(x_{k+1})+g}{\upsilon_k-x_{k+1}}]\\
        &~~~-\tfrac{a_{k+1}^2}{2}\|\nabla f(x_{k+1})+g\|_*^2.
    \end{align*}
    Setting $y_k = x_k+\tau_k(\upsilon_k-x_k)=x_k+\tau_k u_k$ and $\widehat{\tau}_k=\tfrac{a_{k+1}}{A_{k+1}}\in(0,1]$, it can be deduced from \eqref{eq:foOptConds} and \eqref{eq:inprodnablaFY+} that
    \begin{align*}
        A_k F(x_k)&+ a_{k+1}[F(x_{k+1})+\innprod{\nabla f(x_{k+1})+g}{\upsilon_k-x_{k+1}}]\\ 
        &\geq A_{k+1} F(x_{k+1})+ \innprod{\nabla f(x_{k+1})+g}{a_{k+1}\upsilon_k+A_k x_k-A_{k+1}x_{k+1}}\\
        &= A_{k+1} [F(x_{k+1})+ \innprod{\nabla f(x_{k+1})+g}{(\widehat{\tau}_k-\tau_k)u_k+y_k-x_{k+1}}]\\
        &\geq A_{k+1} [F(x_{k+1})+ \innprod{\nabla f(x_{k+1})+g}{y_k-x_{k+1}}]\\
        &= A_{k+1} \left[F(x_{k+1})+ \left(\tfrac{1}{H}\right)^{\tfrac{1}{p}} g_k^{\tfrac{p+1}{p}}\right].
    \end{align*}
    This, the definition of $B_k$, and the identity $\tfrac{a_{k+1}^2}{A_{k+1}+a_{k+1}}=\left(\tfrac{1}{H}\right)^{1/p}g_k^{(1-p)/p}$ eventually lead to
    \begin{align*}
        \Psi_{k+1}(x) \geq A_{k+1} \left[F(x_{k+1})+ \left(\tfrac{1}{H}\right)^{\tfrac{1}{p}} g_k^{\tfrac{p+1}{p}}\right] +B_k -\tfrac{a_{k+1}^2}{2}g_k^2=A_{k+1}F(x_{k+1})+B_{k+1},
    \end{align*}
    implying that \eqref{eq:akBkPsisIneq2} is for $k+1$.
     
    By \eqref{eq:akBkPsisIneq1} and \eqref{eq:akBkPsisIneq2}, we get    
    \begin{equation}\label{eq:AkpsikAkpsi*Ineq0}
        A_k F(x_k)+B_k \leq \Psi_k(x^*)\leq A_k F(x^*)+\tfrac{1}{2} \|x^*-x_0\|^2, \quad x^*\in\dom\psi,
    \end{equation}
    leading to
    \begin{equation}\label{eq:FkFsUB}
        F(x_k)-F(x^*)\leq \tfrac{R_0^2}{2A_k}.
    \end{equation}
    with $R_0=\|x_0-x^*\|$. Further, setting $x=x^*$ in \eqref{eq:AkpsikAkpsi*Ineq0} and using \eqref{eq:Bk1} and $F(x^*)\leq F(x_k)$, it can be concluded that
    \begin{equation}\label{eq:sumAi+1}
        \left(\tfrac{1}{H}\right)^{\tfrac{1}{p}}\sum_{i=0}^{k-1} A_{i+1} g_i^{(p+1)/p} \leq 2\left(A_k(F(x^*)-F(x_k))+\tfrac{1}{2} \|x^*-x_0\|^2\right)\leq R_0^2.
    \end{equation}
    
From \eqref{eq:Ak} and Line~6 of Algorithm~\ref{alg:aihoppaSS}, we obtain
\begin{align*}
A_{k+1}^{1/2}-A_k^{1/2} = \tfrac{a_{k+1}}{A_{k+1}^{1/2}+A_k^{1/2}}= \tfrac{a_{k+1}}{2A_{k+1}^{1/2}}= \tfrac{1}{2} \left(\tfrac{1}{H}\right)^{\tfrac{1}{2p}} g_k^{\tfrac{1-p}{2p}},
\end{align*}    
i.e., setting $\sigma_i=g_{i-1}^{\tfrac{p-1}{2p}}$, it holds that
\begin{align*}
\tfrac{1}{2} \left(\tfrac{1}{H}\right)^{\tfrac{1}{2p}} \sum_{i=1}^k \tfrac{1}{\sigma_i} \leq \sum_{i=1}^k A_i^{1/2}-A_{i-1}^{1/2}=A_k^{1/2},
\end{align*}   
leading to $A_k\geq \tfrac{1}{4} \left(\tfrac{1}{H}\right)^{\tfrac{1}{p}} \left(\sum_{i=1}^k \tfrac{1}{\sigma_i}\right)^2$. Moreover, it follows from \eqref{eq:sumAi+1} that $\sum_{i=1}^k A_i\sigma_i^{\tfrac{2(p+1)}{p-1}}\leq H^{\tfrac{1}{p}} R_0^2$. Now, applying Lemma~\ref{lem:ineqPositNum} with $\gamma=\tfrac{2(p+1)}{p-1}$ and $\omega_i=A_i$ to the latter inequality implies
\begin{align*}
	\sum_{i=1}^k \tfrac{1}{\xi_i}\geq H^{-\tfrac{p-1}{2p(p+1)}}R_0^{-\tfrac{p-1}{p+1}} \left(\sum_{i=1}^k A_i^{\tfrac{p-1}{3p+1}}\right)^{\tfrac{3p+1}{2(p+1)}},
\end{align*}
which leads to
\begin{align*}
	A_k \geq \tfrac{1}{4} H^{-\tfrac{2}{p+1}}R_0^{-\tfrac{2(p-1)}{p+1}} \left(\sum_{i=1}^k A_i^{\tfrac{p-1}{3p+1}}\right)^{\tfrac{3p+1}{p+1}}.
\end{align*}
Setting $\theta=\tfrac{1}{4} H^{-\tfrac{2}{p+1}}R_0^{-\tfrac{2(p-1)}{p+1}}$, $\alpha=\tfrac{p-1}{3p+1}$, and $\beta=\tfrac{3p+1}{p+1}$, Lemma~\ref{lem:ineqPositSeq} indicates that
\begin{align*}
	A_k \geq \left(\tfrac{1}{4}\right)^{\tfrac{p+1}{2}} H^{-1}R_0^{-(p-1)} k^{\tfrac{3p+1}{2}}, \quad k\geq 1.
\end{align*}
Together with \eqref{eq:FkFsUB}, this adjusts \eqref{eq:akBkIneq}.
\end{proof}

\begin{rem}[stopping criteria for Algorithm~\ref{alg:aihoppaSS}]
\label{rem:stopCrit1}
    Let us define the function $\func{\mathcal{L}_k}{\dom\psi}{\mathbb{R}}$ given by
\begin{equation}\label{eq:LkQR}
    \mathcal{L}_k(x)=\tfrac{1}{A_k}\left(\sum_{i=1}^k a_i\left[f(y_i)+\innprod{\nabla f(y_i)}{x-y_i}+\psi(x)\right]\right), \quad Q_R=\set{x\in\dom\psi ~|~ \|x-x_0\|\leq R}, 
\end{equation}
for some $R\geq R_0$. Setting $\mathcal{L}_k^*=\min_{x\in Q_R}\set{\mathcal{L}_k(x)}$, it can be concluded from \eqref{eq:akBkPsisIneq2} that
\begin{align*}
    A_k F(x_k) &\leq A_k F(x_k)+B_k\leq \Psi_k^*=\min_{x\in\dom\psi}\set{A_k \mathcal{L}_k(x)+\tfrac{1}{2}\|x-x_0\|^2}\\
    &\leq \min_{x\in Q_R}\set{A_k \mathcal{L}_k(x)+\tfrac{1}{2}\|x-x_0\|^2}\\
    &\leq \min_{x\in Q_R}\set{A_k \mathcal{L}_k(x)+\tfrac{1}{2} R^2}\\
    & = A_k \mathcal{L}_k^*+\tfrac{1}{2} R^2.
\end{align*}
This consequently leads to the inequality
\begin{align*}
    F(x_k)-\mathcal{L}_k^* \leq \tfrac{1}{2A_k} R^2.
\end{align*}
Since $F(x)\geq\mathcal{L}_k(x)$, we have $F(x^*)\geq\mathcal{L}_k^*$, i.e., $F(x_k)-F(x^*)\leq F(x_k)-\mathcal{L}_k^*\leq \tfrac{1}{2A_k} R^2$. Therefore, $F(x_k)-\mathcal{L}_k^*\leq\varepsilon$ or $A_k\geq R^2/(2\varepsilon)$ yield
\begin{align*}
 F(x_k)-F(x^*)\leq \varepsilon.
\end{align*}
Assuming the efficient computation of $\mathcal{L}_k^*$ or availability of the constant $R\geq R_0$, the inequality $F(x_k)-\mathcal{L}_k^*\leq\varepsilon$  (or $A_k\geq R^2/(2\varepsilon)$) can be used as a stopping criterion for Algorithm~\ref{alg:aihoppaSS}.
\end{rem}
\section{Inexact high-order proximal-point segment search method}\label{sec:ihoppSegmSearch}

As described in Remark~\ref{rem:exactSol} and Example~\ref{exa:exactSol}, finding the exact solution of $p$th-order composite proximal-point segment search can be complicated and costly for arbitrary choices of functions $f(\cdot)$ and $\psi(\cdot)$. As such, we explore approximate solutions of \eqref{eq:proxLS}, study the complexity of finding such approximations, and investigate the global rate of convergence of an inexact accelerated high-order proximal-point methods with the segment search. Our main idea is to first approximate a solution of \eqref{eq:prox00} and find an acceptable $\tau$ by applying a bisection method, which will lead to an inexact solution of \eqref{eq:proxLS}.

Let us begin by describing the set of \textit{acceptable solution} of the problem \eqref{eq:prox00} that is
\begin{equation}\label{eq:A}
     \A(\bar x,\beta)= \set{(x,g)\in\dom \psi\times\E^* ~:~ \|\nabla f_{\bar x, H}^p(x)+g\|_*\leq \beta \|\nabla f(x)+g\|_*},
\end{equation}
for some $g\in\partial \psi(x)\neq \emptyset$ and
\begin{equation}\label{eq:f}
    f_{\bar x, H}^p(x)=f(x)+H d_{p+1}(x-\bar x),
\end{equation}
where $\beta\in[0,1)$ is the tolerance parameter; cf. \cite{ahookhosh2020inexact}. In particular case of $\psi\equiv 0$, the set $ \A(\bar x,\beta)$ leads to inexact solutions for the problem \eqref{eq:prox00}, which was recently studied for smooth unconstrained convex problems in \cite{nesterov2020inexact}. Let us highlight that the set \eqref{eq:A} includes the exact solution of \eqref{eq:prox00} (so it is nonempty); however, the exact solution may be the only solution close enough to the optimizer; see \cite[Example~2.1]{ahookhosh2020inexact}. In  \cite[Section~2.1]{ahookhosh2020inexact}, it was shown that an acceptable solution of $p$th-order composite proximal-point operator \eqref{eq:prox00} can be obtained by applying one step of the $p$th-order tensor method. Moreover, in Section~\ref{sec:supFastSS}, we describe a non-Euclidean proximal method for for finding an inexact solution \eqref{eq:prox00}.

We next present some results of the definition \eqref{eq:A}, which is necessary in the upcoming sections.

\begin{lem}[properties of acceptable solutions] \cite[Lemma~2.2]{ahookhosh2020inexact}\label{lem:solProp}
    Let $(T,g)\in  \A(\bar x,\beta)$ for some $g\in\partial\psi(T)$. Then
    \begin{equation}\label{eq:solProp1}
        (1-\beta)\|\nabla f(T)+g\|_* \leq H \|T-\bar x\|^p \leq (1+\beta)\|\nabla f(T)+g\|_*,
    \end{equation}
    \begin{equation}\label{eq:solProp2}
        \innprod{\nabla f(T)+g}{\bar x-T} \geq \tfrac{H}{1+\beta} \|T-\bar x\|^{p+1}.
    \end{equation}
    If additionally $\beta\leq\tfrac{1}{p}$, then
    \begin{equation}\label{eq:solProp3}
        \innprod{\nabla f(T)+g}{\bar x-T} \geq \left(\tfrac{1-\beta}{H}\right)^{1/p} \|\nabla f(T)+g\|_*^{\tfrac{p+1}{p}}.
    \end{equation}
\end{lem}

We now continue with the following lemma providing some useful properties of the inexact proximal solutions, which is needed in Section~\ref{sec:compAS}.

\begin{lem}[further properties of acceptable solutions]\label{lem:AnormIneq}
    Let $(T,g)\in  \A(\bar x,\beta)$ for some $g\in\partial\psi(T)$ and $\beta\in [0,1/2)$. Then, we have
    \begin{equation}\label{eq:normTx*}
        \|T-x^*\|^2\leq \tfrac{(1-\beta)^2}{1-2\beta} \|\bar x-x^*\|^2.
    \end{equation}
    In particular, if $\beta\leq \tfrac{3}{8}$, then $\|T-x^*\|\leq \tfrac{5}{4} \|\bar x-x^*\|$.
\end{lem}

\begin{proof}
    Let us denote $r= \|T-\bar x\|$. Writing the first-order optimality conditions for \eqref{eq:prox00}, there exists a subgradient $g\in\partial \psi(T)$ such that $g=-\nabla f(T)-H\|T-\bar x\|^{p-1}B(T-\bar x)$, i.e.,
    \begin{align*}
        B(T-\bar x) = \tfrac{1}{H\|T-\bar x\|^{p-1}} \left(-\nabla f(T)-g\right).
    \end{align*}
    On the other hand, the subgradient inequality and $F(x^*)\leq F(T)$ ensure
    \begin{align*}
        \innprod{\nabla f(T)+g}{x^*-T} \leq F(x^*)-F(T)\leq 0 \leq \beta \|\nabla f(T)+g\| \|x^*-T\|.
    \end{align*}
    Then, it follows from the last two inequalities and \eqref{eq:solProp1} that
    \begin{align*}
         \|T-x^*\|^2 &= \|\bar x-x^*\|^2 + 2\innprod{B(T-\bar x)}{(\bar x-T)+(T-x^*)}+r^2\\
         &= \|\bar x-x^*\|^2 + 2\innprod{B(T-\bar x)}{T-x^*}-r^2\\
         &= \|\bar x-x^*\|^2 + \tfrac{2}{Hr^{p-1}}\innprod{-\nabla f(T)-g}{T-x^*}-r^2\\
         &\leq \|\bar x-x^*\|^2 + \tfrac{2\beta}{Hr^{p-1}}\|\nabla f(T)+g\|_*\|T-x^*\|-r^2\\
         &\leq \|\bar x-x^*\|^2 + \max_r\set{\tfrac{2\beta}{1-\beta}\|T-x^*\| r-r^2}\\
         &= \|\bar x-x^*\|^2 + \tfrac{\beta^2}{(1-\beta)^2}\|T-x^*\|^2,
    \end{align*}
    leading to \eqref{eq:normTx*}. Note that the function $\func{\zeta}{\R}{\R}$ given by $\zeta(\beta)=\tfrac{(1-\beta)^2}{1-2\beta}$ is increasing on the interval $[0,1/2)$, i.e., it is bounded above by $25/16$ for $\beta\in\left[0, \tfrac{3}{8}\right]$. We consequently come to the inequality $\|T-x^*\|^2\leq \tfrac{25}{16} \|\bar x-x^*\|^2$, adjusting the result.
\end{proof}

We next consider finding an approximate solution for the minimization problem \eqref{eq:proxLS}. To this end, writing the first-order optimality conditions for \eqref{eq:proxLS}, there exists $g\in\partial \psi(x_{k+1})$ such that
\begin{align}\label{eq:optCondsk1}
    (\tau-\tau_{k+1})\innprod{\nabla f(x_{k+1})+g}{u_k}\geq 0, \quad \forall \tau\in [0,1],
\end{align}
where $u_k=\upsilon_k-x_k$ and the points $x_k,\upsilon_k\in\dom\psi$ are given. Let us assume that $(x_k^0,g)\in\A(x_k,\beta)$ such that $\innprod{\nabla f(x_k^0)+g}{u_k}\geq 0$. Then, it can be concluded from \eqref{eq:optCondsk1} that $\tau_{k+1}=0$, i.e., $(x_{k+1},\tau_{k+1})=(x_k^0,0)\in\sprox(x_k, u_k)$. Moreover, if $(x_k^1,\overline{g})\in\A(\upsilon_k,\beta)$ such that $\innprod{\nabla f(x_k^1)+\overline{g}}{u_k}\leq 0$. Then, \eqref{eq:optCondsk1} yields $\tau_{k+1}=1$, i.e., $(x_{k+1},\tau_{k+1})=(x_k^1,1)\in\sprox(\upsilon_k, u_k)$. Otherwise, we come to the inequalities $\innprod{\nabla f(x_k^0)+g}{u_k}< 0<\innprod{\nabla f(x_k^1)+\overline{g} }{u_k}$. Therefore, there exists $\widehat\tau\in (0,1)$ and $(\widehat x_k,\widehat g)\in\A(x_k+\widehat\tau u_k,\beta)$ such that $\innprod{\nabla f(\widehat x_k)+,\widehat g }{u_k}=0$, i.e., the inequality \eqref{eq:optCondsk1} is consequently satisfied. Since we aim to verify an approximate solution of the the high-order proximal-point minimization with segment search, we can use a root finding technique to find an approximate solution $\tau$ such that 
\begin{equation}
    \innprod{\nabla f(\widetilde x_k)+,\widetilde g }{u_k}\approx 0, \quad (\widetilde x_k,\widetilde g)\in\A(x_k+\widetilde\tau u_k,\beta).
\end{equation}
To do so, we will suggest the bisection scheme in Section~\ref{sec:compAS}.

\vspace{3mm}
\RestyleAlgo{boxruled}
\begin{algorithm}[ht!]
\DontPrintSemicolon \KwIn{$x_{0}\in\dom\psi$,~$\upsilon_0=x_0$,~ $\beta\in [0,3/(3p+2)]$,~ $H>0$,~$A_0=0$,~ $\Psi_0=\tfrac{1}{2}\|x-x_0\|^2$,~ $k=0$;} 
\Begin{ 
    \While {stopping criterion does not hold}{ 
        Set $u_k=\upsilon_k-x_k$, $y_k=x_k$, and compute $(x_k^0,g)\in\A(y_k,\beta)$ for $g\in\partial \psi(x_k^0)$;\;
        \uIf{$\innprod{\nabla f(x_k^0)+g}{u_k}\geq 0$}{
            $\eL_k(x)=\ell_{x_k^0}(x)+\psi(x)$,
            $x_{k+1}=x_k^0,~ g_k=\| \nabla f(x_k^0)+g\|_*$;\;
        }
        \Else{
            $y_k=\upsilon_k$ and compute $(x_k^1,\overline{g})\in\A(y_k,\beta)$ for $\overline{g}\in\partial \psi(x_k^1)$;\;
            \uIf{$\innprod{\nabla f(x_k^1)+\overline{g} }{u_k}\leq 0$}{
                $\eL_k(x)=\ell_{x_k^1}(x)+\psi(x)$,
                $x_{k+1}=x_k^1,~ g_k=\| \nabla f(x_k^1)+\overline{g} \|_*$ for $\overline{g} \in\partial \psi(x_k^1)$;\;
            }
            \Else{
                Find $0\leq\tau_k^1\leq\tau_k^2\leq 1$, $y_k^1=x_k+\tau_k^1u_k$; $(T_k^1,\widehat{g})\in\A(y_k^1,\beta)$ for $\widehat{g}\in\partial \psi(T_k^1)$, $y_k^2=x_k+\tau_k^2u_k$, and $(T_k^2,\widetilde{g})\in\A(y_k^2,\beta)$ for $\widetilde{g}\in\partial \psi(T_k^2)$ such that
                \begin{align}\label{eq:else3}
                    \beta_k^1\leq 0 \leq \beta_k^2,~ \alpha_k(\tau_k^1-\tau_k^2)\beta_k^1\leq \tfrac{1}{2} \left(\tfrac{1-\beta}{H}\right)^{1/p} g_k^{(p+1)/p},~ \alpha_k = \tfrac{\beta_k^2}{\beta_k^2-\beta_k^1}\in [0,1],
                \end{align}
                with $\beta_k^1=\innprod{\nabla f(T_k^1)+\widehat{g}}{u_k}$, $\beta_k^2=\innprod{\nabla f(T_k^2)+\widetilde{g}}{u_k}$, 
                \begin{align*}
                    g_k=\left(\alpha_k\| \nabla f(T_k^1)+\widehat{g}\|_*^{(p+1)/p}+(1-\alpha_k)\| \nabla f(T_k^2)+\widetilde{g}\|_*^{(p+1)/p}\right)^{p/(p+1)};
                \end{align*}
                Set $\eL_k(x)=\alpha_k \ell_{T_k^1}(x) +(1-\alpha_k) \ell_{T_k^2}(x)+\psi(x)$ and $x_{k+1}=\alpha_k T_k^1+(1-\alpha_k)T_k^2$;\; 
            }
        }
        Compute $a_{k+1}$ by solving  $\tfrac{a_{k+1}^2}{A_{k+1}+a_{k+1}}=\tfrac{1}{4}\left(\tfrac{1-\beta}{H}\right)^{1/p}g_k^{(1-p)/p}$, and $A_{k+1}=A_k+a_{k+1}$;\;
        Set $\Psi_{k+1}(x)=\Psi_k(x)+a_{k+1}\eL_k(x)$ and compute $\upsilon_{k+1}=\argmin_{x\in\dom\psi} \Psi_{k+1}(x)$;\;
     } 
}
\caption{Inexact High-Order Proximal-Point Segment Search Algorithm\label{alg:iaihoppaSS}}
\end{algorithm}
\vspace{3mm}

The next result verifies the necessary properties of the linear estimation function $\mathcal{L}_k(\cdot)$.

\begin{lem}[properties of linear estimation function]\label{lem:fkvarphikfk1Ineq}
    Let $\seq{x_k}$ be the generated sequence by Algorithm~\ref{alg:iaihoppaSS}. Then, for all $k\geq 0$, we have
    \begin{equation}\label{eq:psikPsikpsik1Ineq}
        F(x_k)\geq \eL_k(x_k)\geq F(x_{k+1})+\tfrac{1}{2} \left(\tfrac{1-\beta}{H}\right)^{1/p} g_k^{(p+1)/p}.
    \end{equation}
    Moreover, for any $k\geq 0$, we have
    \begin{equation}\label{eq:sumGip1p}
        \sum_{i=k}^\infty g_i^{(p+1)/p} \leq 2\left(\tfrac{H}{1-\beta}\right)^{1/p} (F(x_k)- F^*).
    \end{equation}
\end{lem}

\begin{proof}
    By the definition of $\eL_k(\cdot)$ in Algorithm~\ref{alg:iaihoppaSS} and the subgradient inequality, the left-hand side inequality is valid. To show the second inequality, we consider three cases: (i) $\innprod{\nabla f(x_k^0)+g}{u_k}\geq 0$; (ii) $\innprod{\nabla f(x_k^1)+\overline{g} }{u_k}\leq 0$; (iii) $\innprod{\nabla f(x_k^0)+g}{u_k}< 0$ and $\innprod{\nabla f(x_k^1)+\overline{g} }{u_k}> 0$.
    
    In Case (i), $\innprod{\nabla f(x_k^0)+g}{u_k}\geq 0$, the fact that $x_{k+1}=x_k^0$, $(x_{k+1},g)\in\A(x_k,\beta)$ for $g\in\partial\psi(x_{k+1})$, and \eqref{eq:solProp3} yield
    \begin{align*}
        \eL_k(x_k)&= f(x_{k+1})+\innprod{\nabla f(x_{k+1})}{x_k-x_{k+1}}+\psi(x_k)\\
        &\geq F(x_{k+1})+\innprod{\nabla f(x_{k+1})+g}{x_k-x_{k+1}}\\
        &\geq F(x_{k+1})+\left(\tfrac{1-\beta}{H}\right)^{1/p} \|\nabla f(x_{k+1})+g\|_*^{(p+1)/p}.
    \end{align*}
    In Case (ii), $\innprod{\nabla f(x_k^1)+\overline{g} }{u_k}\leq 0$, we have $x_{k+1}=x_k^1$ and $(x_{k+1},\overline{g})\in\A(\upsilon_k,\beta)$ for $g\in\partial\psi(x_{k+1})$. Then, it follows from \eqref{eq:solProp3} that
    \begin{align*}
        \eL_k(x_k)&= f(x_{k+1})+\innprod{\nabla f(x_{k+1})}{x_k-x_{k+1}}+\psi(x_k)\\
        &\geq F(x_{k+1})+\innprod{\nabla f(x_{k+1})+\overline{g}}{\upsilon_k-u_k-x_{k+1}}\\
        &\geq F(x_{k+1})+\innprod{\nabla f(x_{k+1})+\overline{g}}{\upsilon_k-x_{k+1}}\\
        &\geq F(x_{k+1})+\left(\tfrac{1-\beta}{H}\right)^{1/p} \|\nabla f(x_{k+1})+\overline{g}\|_*^{(p+1)/p}.
    \end{align*}
    In Case (iii), $\innprod{\nabla f(x_k^0)+g}{u_k}< 0$ and $\innprod{\nabla f(x_k^1)+\overline{g} }{u_k}> 0$, there exist $0\leq\tau_k^1\leq\tau_k^2\leq 1$ and $(T_k^1,\widehat{g})\in\A(y_k^1,\beta)$ and $(T_k^2,\widetilde{g})\in\A(y_k^2,\beta)$ such that \eqref{eq:else3} holds. Let us set $\widehat{y}_k=\tfrac{1}{2}(y_k^1+y_k^2)$. Accordingly, it ensures that
    \begin{equation}\label{eq:nablaLkuk}
        \begin{split}
        \innprod{\alpha_k(\nabla f(T_k^1)+\widehat{g})+(1-\alpha_k)(\nabla f(T_k^2)+\widetilde{g})}{u_k}&=\alpha_k\beta_k^1+(1-\alpha_k)\beta_k^2\\ &=\tfrac{\beta_k^2}{\beta_k^2-\beta_k^1}\beta_k^1+(1-\tfrac{\beta_k^2}{\beta_k^2-\beta_k^1})\beta_k^2= 0.
        \end{split}
    \end{equation}
   Combining with the definition of $\eL_k(\cdot)$, $x_k=\widehat{y}_k-\tfrac{\tau_k^1+\tau_k^2}{2}u_k$, $x_{k+1}=\alpha_k T_k^1+(1-\alpha_k)T_k^2$, the convexity of $F(\cdot)$, and \eqref{eq:solProp3}, this implies
    \begin{align*}
        \eL_k(x_k)&= \alpha_k \left[f(T_k^1)+\innprod{\nabla f(T_k^1)}{x_k-T_k^1}\right]+(1-\alpha_k) \left[f(T_k^2)+\innprod{\nabla f(T_k^2)}{x_k-T_k^2}\right]+\psi(x_k)\\
        &\geq \alpha_k\left[F(T_k^1)+\innprod{\nabla f(T_k^1)+\widehat{g}}{x_k-T_k^1}\right]+(1-\alpha_k) \left[F(T_k^2)+\innprod{\nabla f(T_k^2)+\widetilde{g}}{x_k-T_k^2}\right]\\
        &\geq F(x_{k+1})+\alpha_k\innprod{\nabla f(T_k^1)+\widehat{g}}{\widehat{y}_k-T_k^1}+(1-\alpha_k) \innprod{\nabla f(T_k^2)+\widetilde{g}}{\widehat{y}_k-T_k^2}\\
        &= F(x_{k+1})+\alpha_k\innprod{\nabla f(T_k^1)+\widehat{g}}{y_k^1-\tfrac{1}{2}(y_k^1-y_k^2)-T_k^1}\\
        &~~~+(1-\alpha_k) \innprod{\nabla f(T_k^2)+\widetilde{g}}{y_k^2+\tfrac{1}{2}(y_k^1-y_k^2)-T_k^2}\\
        &= F(x_{k+1})+\alpha_k\innprod{\nabla f(T_k^1)+\widehat{g}}{y_k^1-\tfrac{1}{2}(\tau_k^1-\tau_k^2)u_k-T_k^1}\\
        &~~~+(1-\alpha_k) \innprod{\nabla f(T_k^2)+\widetilde{g}}{y_k^2+\tfrac{1}{2}(\tau_k^1-\tau_k^2)u_k-T_k^2}\\
        &\geq F(x_{k+1})+\left(\tfrac{1-\beta}{H}\right)^{1/p} \left(\alpha_k\|\nabla f(T_k^1)+\widehat{g}\|_*^{(p+1)/p}+(1-\alpha_k) \|\nabla f(T_k^2)+\widetilde{g}\|_*^{(p+1)/p}\right)\\
        &~~~-\alpha_k(\tau_k^1-\tau_k^2)\beta_k^1+\tfrac{1}{2}(\tau_k^1-\tau_k^2)[\alpha_k\beta_k^1+(1-\alpha_k)\beta_k^2].
    \end{align*}
    This, $\alpha_k\beta_k^1+(1-\alpha_k)\beta_k^2=0$, and $\alpha_k(\tau_k^1-\tau_k^2)\beta_k^1\leq \tfrac{1}{2} \left(\tfrac{1-\beta}{H}\right)^{1/p} g_k^{(p+1)/p}$ yield
    \begin{align*}
        \eL_k(x_k)
        &\geq F(x_{k+1})+\left(\tfrac{1-\beta}{H}\right)^{1/p} g_k^{(p+1)/p}      -\alpha_k(\tau_k^1-\tau_k^2)\beta_k^1,
    \end{align*}
    implying that \eqref{eq:psikPsikpsik1Ineq} holds. Summing both sides of \eqref{eq:psikPsikpsik1Ineq} from $i=k$ to $+\infty$ leads to our desired inequality \eqref{eq:sumGip1p}.
\end{proof}

Let us define
\begin{align*}
    D_0 = \max_{x\in\dom\psi} \set{\|x-x_0\| ~:~ F(x)\leq F(x_0)},\quad D_* = \max_{x\in\dom\psi} \set{\|x-x^*\| ~:~ F(x)\leq F(x_0)},
\end{align*}
which satisfies
\begin{equation}
    D_0\leq 2D_*.
\end{equation}
We next verify main properties of the estimating sequence $\seq{\Psi_k(\cdot)}$ in Algorithm \ref{alg:iaihoppaSS}.

\begin{lem}[estimating sequence properties]\label{lem:estSeqBkIneq}
    Let the sequence $\seq{x_k}$ generated by Algorithm \ref{alg:iaihoppaSS} be well defined. Then, for all $k\geq 0$, we have
    \begin{align}
        &\Psi_k(x)\leq A_k F(x)+\tfrac{1}{2} \|x-x_0\|^2,\quad x\in\E,\label{eq:estSeqBkIneq1}\\
        &A_k F(x_k)+B_k \leq \Psi_k^*=\min_{x\in\dom\psi} \Psi_k(x),\label{eq:estSeqBkIneq2}
    \end{align}
    where 
    \begin{equation}\label{eq:Bk2}
        B_k=\tfrac{1}{4} \left(\tfrac{1-\beta}{H}\right)^{1/p} \sum_{i=0}^{k-1} A_{i+1} g_i^{(p+1)/p}.
    \end{equation}
    Moreover, we have
    \begin{equation}\label{eq:normxkx*D*vkx*D*}
        \|\upsilon_k-x_k\|\leq D_0,\quad \|x_k-x^*\|\leq D_*, \quad \|\upsilon_k-x^*\|\leq D_*.
    \end{equation}
\end{lem}

\begin{proof}
    The proof is given by induction on $k$. For $k=0$, $\Psi_0=\tfrac{1}{2} \|x-x_0\|^2$, $A_0=B_0=\Psi_0^*=0$, and so both \eqref{eq:estSeqBkIneq1} and \eqref{eq:estSeqBkIneq2} hold. Assuming both of the inequalities \eqref{eq:estSeqBkIneq1} and \eqref{eq:estSeqBkIneq2} for $k$, we show them for $k+1$. Considering Lines 4--13 of Algorithm~\ref{alg:iaihoppaSS}, three cases are recognized, namely, (i) $\innprod{\nabla f(x_k^0)+g}{u_k}\geq 0$; (ii) $\innprod{\nabla f(x_k^1)+\overline{g} }{u_k}\leq 0$; (iii) $\innprod{\nabla f(x_k^0)+g}{u_k}< 0$ and $\innprod{\nabla f(x_k^1)+\overline{g} }{u_k}> 0$.

    In all above cases, we have that $\eL_k(x)\leq F(x)$, as is clear from the definition of $\eL_k(\cdot)$ in Algorithm~\ref{alg:iaihoppaSS}. Then, it follows from the induction assumption that
    \begin{align*}
        \Psi_{k+1}(x)&=\Psi_k(x)+a_{k+1} \eL_k(x)
        \leq A_k F(x)+ \tfrac{1}{2} \|x-x_0\|^2+a_{k+1} \eL_k(x)\\
        &\leq A_{k+1} F(x)+ \tfrac{1}{2} \|x-x_0\|^2.
    \end{align*}
    Moreover, since $\Psi_k(\cdot)$ is $1$-strongly convex, the inequality \eqref{eq:estSeqBkIneq2} ensures that
    \begin{align*}
        \Psi_{k+1}(x) &= \Psi_k(x)+a_{k+1} \eL_k(x)\\
        &\geq \Psi_k(\upsilon_k)+\innprod{\nabla \Psi_k(\upsilon_k)}{x-\upsilon_k}+\tfrac{1}{2} \|x-\upsilon_k\|^2+a_{k+1} \eL_k(x)\\
        &\geq A_k F(x_k)+B_k +\tfrac{1}{2}\|x-\upsilon_k\|^2 +a_{k+1} \eL_k(x).
    \end{align*}
    For Case (i), $\innprod{\nabla f(x_{k+1})+g}{u_k}\geq 0$, together with the subgradient inequality for $\eL_k(\cdot)$, \eqref{eq:psikPsikpsik1Ineq}, \eqref{eq:nablaLkuk}, and $x_k=\upsilon_k-u_k$, the last inequality implies that 
    \begin{align*}
        \Psi_{k+1}(x) &\geq A_k F(x_k) +a_{k+1}[\eL_k(x_k)+\innprod{\nabla f(x_{k+1})+g}{x-x_k}]+B_k +\tfrac{1}{2}\|x-\upsilon_k\|^2\\
        &\geq A_{k+1} F(x_{k+1})+a_{k+1}\innprod{\nabla f(x_{k+1})+g}{x-x_k} \\
        &~~~+\tfrac{1}{2}A_{k+1}\left(\tfrac{1-\beta}{H}\right)^{1/p} g_k^{(p+1)/p}+B_k +\tfrac{1}{2}\|x-\upsilon_k\|^2\\
        &\geq A_{k+1} F(x_{k+1})+a_{k+1}\innprod{\nabla f(x_{k+1})+g}{x-\upsilon_k}\\
        &~~~ +\tfrac{1}{2}A_{k+1}\left(\tfrac{1-\beta}{H}\right)^{1/p} g_k^{(p+1)/p}+B_k +\tfrac{1}{2}\|x-\upsilon_k\|^2.
    \end{align*}
    In Case (ii), $\innprod{\nabla f(x_{k+1})+\overline{g}}{u_k}\leq 0$, it follows from the subgradient inequality and \eqref{eq:solProp3} that
    \begin{align*}
        \Psi_{k+1}(x) &\geq A_k F(x_k) +a_{k+1}[\eL_k(x_k)+\innprod{\nabla f(x_{k+1})+\overline{g}}{x-x_k}]+B_k +\tfrac{1}{2}\|x-\upsilon_k\|^2\\
        &\geq A_{k+1} F(x_{k+1})+\innprod{\nabla f(x_{k+1})+\overline{g}}{A_{k+1}(x_k-x_{k+1})
        +a_{k+1}(x-x_k)}\\
        &~~~+B_k+\tfrac{1}{2}\|x-\upsilon_k\|^2\\
         &\geq A_{k+1} F(x_{k+1})+\innprod{\nabla f(x_{k+1})+\overline{g}}{A_{k+1}(\upsilon_k-x_{k+1})
         +a_{k+1}(x-\upsilon_k)}\\
         &~~~+B_k+\tfrac{1}{2}\|x-\upsilon_k\|^2\\
         &\geq A_{k+1} F(x_{k+1})+a_{k+1}\innprod{\nabla f(x_{k+1})+\overline{g}}{x-\upsilon_k}
         +A_{k+1}\left(\tfrac{1-\beta}{H}\right)^{1/p} g_k^{(p+1)/p}\\
         &~~~+B_k +\tfrac{1}{2}\|x-\upsilon_k\|^2.
    \end{align*}
    For Case (iii), $\innprod{\nabla f(x_k^0)+g}{u_k}< 0$ and $\innprod{\nabla f(x_k^1)+\overline{g} }{u_k}> 0$, together with the subgradient inequality for $\eL_k(\cdot)$, \eqref{eq:psikPsikpsik1Ineq}, \eqref{eq:nablaLkuk}, and $x_k=\upsilon_k-u_k$, the last inequality ensures that 
    \begin{align*}
        \Psi_{k+1}(x) &\geq A_k F(x_k)+a_{k+1}\eL_k(x_k)+a_{k+1} \innprod{\alpha_k(\nabla f(T_k^1)+\widehat{g})+(1-\alpha_k)(\nabla f(T_k^2)+\widetilde{g})}{x-x_k}\\
        &~~~+B_k +\tfrac{1}{2}\|x-\upsilon_k\|^2\\
        &\geq A_{k+1} F(x_{k+1})+a_{k+1} \innprod{\alpha_k(\nabla f(T_k^1)+\widehat{g})+(1-\alpha_k)(\nabla f(T_k^2)+\widetilde{g})}{x-x_k}\\ &~~~+\tfrac{1}{2}A_{k+1}\left(\tfrac{1-\beta}{H}\right)^{1/p} g_k^{(p+1)/p}+B_k +\tfrac{1}{2}\|x-\upsilon_k\|^2\\
        &\geq A_{k+1} F(x_{k+1})+a_{k+1} \innprod{\alpha_k(\nabla f(T_k^1)+\widehat{g})+(1-\alpha_k)(\nabla f(T_k^2)+\widetilde{g})}{x-\upsilon_k}\\ &~~~+\tfrac{1}{2}A_{k+1}\left(\tfrac{1-\beta}{H}\right)^{1/p} g_k^{(p+1)/p}+B_k +\tfrac{1}{2}\|x-\upsilon_k\|^2.
    \end{align*}
    Let us set
    \begin{equation}\label{eq:Gk}
        \mathcal{G}_k = \left\{
        \begin{array}{ll}
            \nabla f(x_k^0)+g & \mathrm{if}~ \innprod{\nabla f(x_k^0)+g}{u_k}\geq 0, \\
            \nabla f(x_k^1)+\overline{g} & \mathrm{if}~ \innprod{\nabla f(x_k^1)+\overline{g}}{u_k}\leq 0, \\
            \alpha_k(\nabla f(T_k^1)+\widehat{g})+(1-\alpha_k)(\nabla f(T_k^2)+\widetilde{g}) & \mathrm{if}~ \innprod{\nabla f(x_k^0)+g}{u_k}< 0~ \&~ \innprod{\nabla f(x_k^1)+\overline{g} }{u_k}> 0.
        \end{array}
        \right.
    \end{equation}
    For Cases (i) and (ii), it is clear that $g_k=\|\mathcal{G}_k\|$. On the other hand, for Case (iii), it follows from the inequality $a+b\leq 2^{(t-1)/t}(a^t+b^t)^{1/t}$ for $a,b\geq 0$ and $t\geq 1$ that
    \begin{equation}\label{eq:normGkIneq}
        \begin{split}
        \|\mathcal{G}_k\| &= \alpha_k \|\nabla f(T_k^1)+\widehat{g}\|+(1-\alpha_k)\|\nabla f(T_k^2)+\widetilde{g}\|\\
        &\leq 2^{\tfrac{1}{p+1}}\left(\alpha_k^{\tfrac{p+1}{p}} \|\nabla f(T_k^1)+\widehat{g}\|^{\tfrac{p+1}{p}} +(1-\alpha_k)^{\tfrac{p+1}{p}} \|\nabla f(T_k^2)+\widetilde{g}\|^{\tfrac{p+1}{p}}\right)^{\tfrac{p}{p+1}}\\
        &\leq 2^{\tfrac{1}{p+1}}\left(\alpha_k \|\nabla f(T_k^1)+\widehat{g}\|^{\tfrac{p+1}{p}} +(1-\alpha_k) \|\nabla f(T_k^2)+\widetilde{g}\|^{\tfrac{p+1}{p}}\right)^{\tfrac{p}{p+1}} =  2^{\tfrac{1}{p+1}} g_k.
        \end{split}
    \end{equation}
    Therefore, in all three cases, we have $\|\mathcal{G}_k\|^2 \leq 2^{\tfrac{2}{p+1}} g_k^2\leq 2 g_k^2$. This, the above inequalities for Cases (i)--(iii), and Line 14 of Algorithm \ref{alg:iaihoppaSS} ensure that
    \begin{align*}
        \Psi_{k+1}(x) 
        &\geq A_{k+1} F(x_{k+1})+B_k+\tfrac{1}{2}A_{k+1}\left(\tfrac{1-\beta}{H}\right)^{\tfrac{1}{p}} g_k^{\tfrac{p+1}{p}}+\min_{x\in\E} \set{a_{k+1}\innprod{\mathcal{G}_k}{x-\upsilon_k}+\tfrac{1}{2}\|x-\upsilon_k\|^2}\\
        &\geq A_{k+1} F(x_{k+1})+B_k+\tfrac{1}{2}A_{k+1}\left(\tfrac{1-\beta}{H}\right)^{\tfrac{1}{p}} g_k^{\tfrac{p+1}{p}}-\tfrac{a_{k+1}^2}{2}\|\mathcal{G}_k \|^2\\
        &\geq A_{k+1} F(x_{k+1})+B_k+\tfrac{1}{2}A_{k+1}\left(\tfrac{1-\beta}{H}\right)^{\tfrac{1}{p}} g_k^{\tfrac{p+1}{p}}-a_{k+1}^2 g_k^2\\
        &= A_{k+1} F(x_{k+1})+B_{k+1}+\tfrac{1}{4}A_{k+1}\left(\tfrac{1-\beta}{H}\right)^{\tfrac{1}{p}} g_k^{\tfrac{p+1}{p}}-a_{k+1}^2 g_k^2\\
        &= A_{k+1} F(x_{k+1})+B_{k+1},
    \end{align*}
    leading to \eqref{eq:estSeqBkIneq2}.
    
    In view of the $1$-strongly convexity of $\Psi_k(\cdot)$, \eqref{eq:estSeqBkIneq1}, and \eqref{eq:estSeqBkIneq2}, we have    
    \begin{equation}\label{eq:AkpsikAkpsi*Ineq}
        A_k F(x_k)+B_k+\tfrac{1}{2} \|x-\upsilon_k\|^2 \leq \Psi_k(x)\leq A_k F(x)+\tfrac{1}{2} \|x-x_0\|^2, \quad \forall x\in\dom\psi.
    \end{equation}
    Setting $x=x_k$ implies $\|\upsilon_k-x_k\|\leq D_0$. If we set $x=x_0$, then we have $F(x_k)\leq F(x_0)$, i.e., $\|x_k-x_0\|\leq D_0$ and $\|x_k-x^*\|\leq D_*$. If we set $x=x^*$, then
    \begin{align*}
        \tfrac{1}{2} \|x^*-\upsilon_k\|^2\leq  B_k+\tfrac{1}{2} \|x^*-\upsilon_k\|^2 \leq A_k (F(x^*)-F(x_k))+\tfrac{1}{2} \|x^*-x_0\|^2 \leq \tfrac{1}{2} \|x^*-x_0\|^2,
    \end{align*}
    implying $\|\upsilon_k-x^*\|\leq \|x^*-x_0\| \leq D_*$.
\end{proof}

The following result gives the convergence rate $\mathcal{O}(k^{-(3p+1)/2})$ for the sequence $\seq{x_k}$ generated by Algorithm~\ref{alg:iaihoppaSS} in terms of the function values variation.

\begin{thm}[convergence rate of Algorithm~\ref{alg:iaihoppaSS}]\label{thm:convRateAlg6}
    Let the sequence $\seq{x_k}$ generated by Algorithm~\ref{alg:iaihoppaSS} be well defined. Then, for all $k\geq 0$, we have
    \begin{equation}\label{eq:estSeqBkIneqAlg6}
        F(x_k)-F(x^*)\leq \tfrac{4^p H R_0^{p+1}}{1-\beta} \left(1+\tfrac{2(k-1)}{p+1}\right)^{-\tfrac{3p+1}{2}},
    \end{equation}
    where $R_0=\|x^*-x_0\|$.
\end{thm}

\begin{proof}
    Setting $x=x^*$ in the inequality \eqref{eq:AkpsikAkpsi*Ineq} clearly leads to
    \begin{equation*}
        F(x_k)-F(x^*)\leq \tfrac{R_0^2}{2A_k}.
    \end{equation*}
    Further, in the same way and by using \eqref{eq:Bk1}, it can be concluded that
    \begin{align*}
        \left(\tfrac{1}{H}\right)^{1/p} \sum_{i=0}^{k-1} A_{i+1} g_i^{(p+1)/p} \leq 2\left[A_k(F(x^*)-F(x_k))+\tfrac{1}{2} \|x^*-x_0\|^2\right]\leq R_0^2.
    \end{align*}
    Letting $k\to\infty$ yields
    \begin{equation*}
        \sum_{i=0}^{\infty} A_{i+1} g_i^{(p+1)/p}\leq H^{1/p}R_0^2.
    \end{equation*}
    The remainder of the proof is the same as those of Theorem~\ref{thm:aihoppaSSConvRate}.
\end{proof}

We next verify the convergence rate of the sequence $\seq{x_k}$ in terms of the norm of subgradients of the objective function $F(\cdot)$.

\begin{thm}\label{thm:gt*}
    Let the sequence $\seq{x_k}$ be generated by Algorithm \ref{alg:iaihoppaSS}. Then, for any $t=3k-2$ with $k\geq 0$, we have
    \begin{equation}\label{eq:gt*}
        \|\mathcal{G}_t^*\|  \leq 2 (4^p) (p+1)^{\tfrac{p(3p+1)}{2(p+1)}} \tfrac{HR_0^p}{1-\beta} \left(\tfrac{1}{t-k}\right)^{\tfrac{3}{2}p},
    \end{equation}
    where $\|\mathcal{G}_t^*\|=\min_{0\leq i\leq t} \|\mathcal{G}_i\|$, $\mathcal{G}_i$ is defined in \eqref{eq:Gk}, and $R_0=\|x^*-x_0\|$.
\end{thm}

\begin{proof}
    From Lines 4-13 of Algorithm \ref{alg:iaihoppaSS}, we recognize three cases: (i) $\innprod{\nabla f(x_k^0)+g}{u_k}\geq 0$; (ii) $\innprod{\nabla f(x_k^1)+\overline{g} }{u_k}\leq 0$; (iii) $\innprod{\nabla f(x_k^0)+g}{u_k}< 0$ and $\innprod{\nabla f(x_k^1)+\overline{g} }{u_k}> 0$. In Cases (i) and (ii), we have $\|\mathcal{G}_k\|=g_k$, for the sequence $\seq{\mathcal{G}_k}$ given by \eqref{eq:Gk}. In Case (iii), the inequality \eqref{eq:normGkIneq} implies
    \begin{align*}
        \|\mathcal{G}_k\| \leq 2^{\tfrac{1}{p+1}}\left(\alpha_k \|\nabla f(T_k^1)+\widehat{g}\|^{\tfrac{p+1}{p}} +(1-\alpha_k) \|\nabla f(T_k^2)+\widetilde{g}\|^{\tfrac{p+1}{p}}\right)^{\tfrac{p}{p+1}} \leq 2 g_k.
    \end{align*}
    Setting $g_t^* =\min_{0 \leq i \leq t} g_i$ and invoking  \eqref{eq:sumGip1p} and \eqref{eq:estSeqBkIneqAlg6} lead to
    \begin{align*}
        \|\mathcal{G}_k^*\|^{\tfrac{p+1}{p}} &\leq  (2g_k^*)^{\tfrac{p+1}{p}} \leq \tfrac{1}{t-k}2^{\tfrac{p+1}{p}}\sum_{i=k+1}^t g_i^{\tfrac{p+1}{p}} \leq \tfrac{1}{t-k}  2^{\tfrac{p+1}{p}} \left(\tfrac{H}{1-\beta}\right)^{\tfrac{1}{p}} (F(x_{k})-F^*) \\
        &\leq \tfrac{1}{t-k} 2^{\tfrac{p+1}{p}}\left(\tfrac{H}{1-\beta}\right)^{\tfrac{1}{p}} \tfrac{4^{p+1} H R_0^{p+1}}{1-\beta} \left(1+\tfrac{2(k-1)}{p+1}\right)^{-\tfrac{3p+1}{2}}\\
        &\leq \tfrac{1}{t-k} 2^{\tfrac{p+1}{p}}\left(\tfrac{H}{1-\beta}\right)^{\tfrac{1}{p}} \tfrac{4^{p+1} H R_0^{p+1}}{1-\beta} \left(\tfrac{p+1}{2(k-1)}\right)^{\tfrac{3p+1}{2}}\\
        &\leq \tfrac{2^{\tfrac{p+1}{p}} 4^{p+1} H^{\tfrac{p+1}{p}} R_0^{p+1}(p+1)^{\tfrac{3p+1}{2}}}{(1-\beta)^{\tfrac{p+1}{p}}} \left(\tfrac{1}{t-k}\right)^{\tfrac{3(p+1)}{2}},
    \end{align*}
    giving \eqref{eq:gt*}.
\end{proof}

\subsection{Implementable stopping criterion for Algorithm~\ref{alg:iaihoppaSS}} \label{sec:imStopCrit}
Here, we verify two stopping criteria for Algorithm~\ref{alg:iaihoppaSS}.
We note that the estimating sequence employed in Algorithm~\ref{alg:iaihoppaSS} is different with respect to the common one used in Algorithm~\ref{alg:aihoppaSS}. Hence, we translate Remark~\ref{rem:stopCrit1} in the following lemma, which ensures that either of $F(x_k)-\mathcal{L}_k^*\leq\varepsilon$ and $A_k\geq R^2/(2\varepsilon)$ can be used as a stopping criterion for Algorithm~\ref{alg:iaihoppaSS}. 

\begin{lem}\label{lem:ineqFkLkhat}
    Let the sequence $\seq{x_k}$ be generated by Algorithm~\ref{alg:iaihoppaSS} and let $\func{\widehat{\mathcal{L}}_k}{\E}{\mathbb{R}}$ be the function given by
    \begin{align*}
        \widehat{\mathcal{L}}_k(x) =\tfrac{1}{A_k} \sum_{i=1}^k \mathcal{L}_i(x)
    \end{align*}
    and $Q_R=\set{x\in\dom\psi ~|~ \|x-x_0\|\leq R}$. If $\widehat{\mathcal{L}}_k^*=\min_{x\in Q_R} \widehat{\mathcal{L}}_k(x)$, then
    \begin{equation}\label{eq:ineqFkLkhat}
         F(x_k)-\widehat{\mathcal{L}}_k^* \leq \tfrac{1}{2A_k} R^2.
    \end{equation} 
\end{lem}

\begin{proof}
    From \eqref{eq:estSeqBkIneq2} and the definition of $\widehat{\mathcal{L}}_k(\cdot)$, we obtain
\begin{align*}
    A_k F(x_k) &\leq A_k F(x_k)+B_k\leq \Psi_k^*=\min_{x\in\dom\psi}\set{A_k \widehat{\mathcal{L}}_k(x)+\tfrac{1}{2}\|x-x_0\|^2}\\
    &\leq \min_{x\in Q_R}\set{A_k \widehat{\mathcal{L}}_k(x)+\tfrac{1}{2}\|x-x_0\|^2}\\
    &\leq \min_{x\in Q_R}\set{A_k \widehat{\mathcal{L}}_k(x)+\tfrac{1}{2} R^2}\\
    & = A_k \widehat{\mathcal{L}}_k^*+\tfrac{1}{2} R^2,
\end{align*}
which establishes \eqref{eq:ineqFkLkhat}.
\end{proof}

Note that $F(x)\geq\widehat{\mathcal{L}}_k(x)$ implying $F(x^*)\geq\widehat{\mathcal{L}}_k^*$, i.e., $F(x_k)-F(x^*)\leq F(x_k)-\widehat{\mathcal{L}}_k^*\leq \tfrac{1}{2A_k} R^2$. As such, $F(x_k)-\widehat{\mathcal{L}}_k^*\leq \varepsilon$ or $A_k\geq R^2/(2\varepsilon)$ suggests $F(x_k)-F(x^*)\leq \varepsilon$, for the accuracy parameter $\varepsilon>0$. Hence, one can employ either $F(x_k)-\widehat{\mathcal{L}}_k^*\leq\varepsilon$ or $A_k\geq R^2/(2\varepsilon)$ as a stopping criterion for Algorithm~\ref{alg:iaihoppaSS}.

We next investigate an inequality based on {\it primal-dual gap}, which suggests a new stopping criterion. Let us define $F^*=F(x^*)$, where $F(\cdot)$ is the composite function defined in \eqref{eq:prob}. Let us assume that the objective function $f(\cdot)$ has the structure $f(x)=\max_{u\in Q_d} \set{\innprod{Ax}{u}-\varphi(u)}$, i.e.,
\begin{align*}
    F(x)=\max_{u\in Q_d} \set{\innprod{Ax}{u}-\varphi(u)}+\psi(x),
\end{align*}
for some convex set $Q_d$ and a $q$-strongly convex function $\func{\varphi}{\E}{\mathbb{R}}$, which is
\begin{align*}
    \innprod{\nabla\varphi(x)-\nabla\varphi(y)}{x-y}\geq \sigma_{\varphi} \|x-y\|^q,
\end{align*}
for $\sigma_{\varphi}>0$. We note that the function $f(\cdot)$ is differentiable with H\"older continuous gradients 
\begin{equation}\label{eq:gradf}
    \nabla f(x)=A^*u(x),
\end{equation}
where H\"older's constants are $\nu=1/(q-1)$ and $L_\nu=(1/\sigma_\varphi)^\nu \|A\|^{1+\nu}$; see \cite[Lemma~6.4.3]{nesterov2018lectures}. 

\begin{prop}[primal-dual gap]\label{prop:primDualGap}
    Let the sequence $\seq{x_k}$ be generated by Algorithm~\ref{alg:iaihoppaSS}. Then, for all $k\geq 0$, we have
    \begin{equation}\label{eq:FkPhikIneq00}
        0\leq (F(x_k)-F^*)+(\phi^*-\phi(\widehat{u}_k))\leq F(x_k)-\phi(\widehat{u}_k))\leq \tfrac{1}{2A_k}R_0^2,
    \end{equation}
    where $\widehat{u}_k$ is the dual solution, $R_0=\|x_0-x^*\|$ and
    \begin{align*}
        \phi^*=\max_{u\in Q_d} \phi(u), \quad \phi(u)=\max_{x\in \dom\psi} \set{\innprod{Ax}{u}-\varphi(u)}+\psi(x). 
    \end{align*}
\end{prop}

\begin{proof}
    By the definition of $\widehat{\mathcal{L}}_k(\cdot)$, we get
    \begin{align*}
        \widehat{\mathcal{L}}_k(x) = \tfrac{1}{A_k} \sum_{i\in I_1} a_i[\ell_{x_{i+1}}(x)+\psi(x)]+\sum_{i\in I_2} a_i[\alpha_i \ell_{T_i^1}(x)+(1-\alpha_i) \ell_{T_i^2}(x)+\psi(x)],
    \end{align*}
    where 
    \begin{align}
        & I_1 = \set{i\in\set{1,\ldots,k} ~|~ \innprod{\nabla f(x_i^0)+g}{u_i}\geq 0~\mathrm{or}~\innprod{\nabla f(x_i^1)+\overline{g}}{u_i}\leq 0},\\
        & I_2 = \set{i\in\set{1,\ldots,k} ~|~ \innprod{\nabla f(x_i^0)+g}{u_i}< 0 ~\&~ \innprod{\nabla f(x_i^1)+\overline{g} }{u_i}> 0}.
    \end{align}
    It follows from \eqref{eq:gradf} that $\nabla f(x_{i+1}) = A^*u(x_{i+1})$, $\nabla f(T_i^1) = A^*u(T_i^1)$, and $\nabla f(T_i^1) = A^*u(T_i^1)$,
    which ensures
    \begin{align*}
        \min_{x\in\E} &\widehat{\mathcal{L}}_k(x) = \min \set{\tfrac{1}{A_k} \left(\sum_{i\in I_1} a_i[\ell_{x_{i+1}}(x)+\psi(x)]+\sum_{i\in I_2} a_i[\alpha_i \ell_{T_i^1}(x)+(1-\alpha_i) \ell_{T_i^2}(x)+\psi(x)]\right)}\\
        &= \min_{x\in\E} \left\{\tfrac{1}{A_k} \sum_{i\in I_1} a_i[\max_{u\in Q_d} \set{\innprod{Ax_{i+1}}{u(x_{i+1})}-\varphi(u(x_{i+1}))} + \innprod{A^*u(x_{i+1})}{x-x_{i+1}}+\psi(x)]\right. \\
        &~~~+\tfrac{1}{A_k} \sum_{i\in I_2} a_i \left[\alpha_i \left(\max_{u\in Q_d} \set{\innprod{AT_i^1}{u(T_i^1)}-\varphi(u(T_i^1))} + \innprod{A^*u(T_i^1)}{x-T_i^1}\right)+\psi(x)\right]\\
        &~~~\left.+\tfrac{1}{A_k} \sum_{i\in I_2} a_i \left[(1-\alpha_i) \left(\max_{u\in Q_d} \set{\innprod{AT_i^2}{u(T_i^2)}-\varphi(u(T_i^2))} + \innprod{A^*u(T_i^2)}{x-T_i^2}\right)+\psi(x)\right]\right\}.
    \end{align*}
    Hence, the subgradient inequality implies
    \begin{align*}
        \min_{x\in\E} \widehat{\mathcal{L}}_k(x) 
        &\leq \min_{x\in\E} \left\{\tfrac{1}{A_k} \sum_{i=1}^k a_i[\max_{u\in Q_d} \set{\innprod{Ax}{u((x_{i+1})}-\varphi(u((x_{i+1}))} +\psi(x)]\right. \\
        &~~~+\tfrac{1}{A_k} \sum_{i\in I_2} a_i \left[\alpha_i \max_{u\in Q_d} \set{\innprod{Ax}{u(T_i^1)}-\varphi(u(T_i^1))} +\psi(x)\right]\\
        &~~~\left.+\tfrac{1}{A_k} \sum_{i\in I_2} a_i \left[(1-\alpha_i) \max_{u\in Q_d} \set{\innprod{Ax}{u(T_i^2)}-\varphi(u(T_i^2))} +\psi(x)\right]\right\}\\
        &\leq -\varphi(\widehat{u}_k) +\min_{x\in\E} \set{\innprod{Ax}{\widehat{u}_k} +\psi(x)}\\
        &= -\varphi(\widehat{u}_k)+\psi_*(-A^*\widehat{u}_k),
    \end{align*}
    where $\widehat{u}_k$ is the dual solution. 
    
    On the other hand, from the definition of estimating sequence defined in Algorithm~\ref{alg:iaihoppaSS} and \eqref{eq:estSeqBkIneq2}, we obtain
    \begin{align*}
        F(x_k) &\leq \min_{x\in\dom\psi} \set{\widehat{\mathcal{L}}_k(x) +\tfrac{1}{2A_k} \|x-x_0\|^2} \leq \min_{x\in\dom\psi} \set{\widehat{\mathcal{L}}_k(x)}+\tfrac{R_0^2}{2A_k}\\ 
        &\leq -\varphi(\widehat{u}_k)+\psi_*(-A^*\widehat{u}_k) +\tfrac{R_0^2}{2A_k},
    \end{align*}
    Together with $\phi^*\leq F^*$, this implies \eqref{eq:FkPhikIneq00}.
\end{proof}

We note that Proposition~\ref{prop:primDualGap} ensures that primal-dual gap has the upper bound $R_0^2/2A_k$. As a result, one can use the inequality $A_k\geq (R_0^2/2\varepsilon)$ as a stopping criterion for Algorithm \ref{alg:iaihoppaSS}.

\subsection{Bisection method for auxiliary segment search}\label{sec:compAS}
This section concerns with an adaptation of the bisection scheme in \cite[Section 4]{nesterov2020inexactSS} to handle the segment search of  Algorithm \ref{alg:iaihoppaSS} for composite minimization \eqref{eq:prob}.

\vspace{3mm}
\RestyleAlgo{boxruled}
\begin{algorithm}[H]
\DontPrintSemicolon \KwIn{$x_k, u_k\in\dom\psi, i=0, \tau_{k,0}^1=0, \beta_{k,0}^1=\beta_k^1, T_{k,0}^1=x_k^0, \beta_{k,0}^2=\beta_k^2, T_{k,0}^2=x_k^1,\alpha_{k,0}=\tfrac{\beta_{k,0}^2}{\beta_{k,0}^2-\beta_{k,0}^1}$,  $g_{k,0}=\left(\alpha_{k,0}\| \nabla f(x_k^0)+g\|_*^{(p+1)/p}+(1-\alpha_k)\| \nabla f(x_k^1)+\overline{g} \|_*^{(p+1)/p}\right)^{p/(p+1)}$;} 
\Begin{ 
    \While {$\alpha_{k,i}(\tau_{k,i}^1-\tau_{k,i}^2)\beta_{k,i}^1> \tfrac{1}{2} \left(\tfrac{1-\beta}{H}\right)^{1/p} g_{k,i}^{(p+1)/p}$}{ 
        Set $\tau_{k,i}^+=\tfrac{1}{2}(\tau_{k,i}^1+\tau_{k,i}^2)$;\; 
        Compute $(T_{k,i}^+,g^+)\in\A(y_{k,i},\beta)$ for $y_{k,i}=x_k+\tau_{k,i}^+u_k$ and $g^+\in\partial \psi(T_{k,i}^+)$;\;
        Set $\beta_{k,i}^1=\innprod{\nabla f(T_{k,i}^+)+g^+}{u_k}$;\;
        \uIf{$\beta_{k,i}^1\leq 0$}{
            $\tau_{k,i+1}^1=\tau_{k,i}^+, \beta_{k,i}^1=\beta_{k,i+1}^+,~ T_{k,i+1}^1=T_{k,i}^+$ and $\tau_{k,i+1}^2=\tau_{k,i}^2, \beta_{k,i}^2=\beta_{k,i+1}^2,~ T_{k,i+1}^2=T_{k,i}^2$;\;
        }
        \Else{
            $\tau_{k,i+1}^1=\tau_{k,i}^1, \beta_{k,i}^1=\beta_{k,i+1}^1,~ T_{k,i+1}^1=T_{k,i}^1$ and $\tau_{k,i+1}^2=\tau_{k,i}^+, \beta_{k,i}^2=\beta_{k,i+1}^+,~ T_{k,i+1}^2=T_{k,i}^+$;\;
        }
        Set $\alpha_{k,i+1}=\tfrac{\beta_{k,i+1}^2}{\beta_{k,i+1}^2-\beta_{k,i+1}^1}\in[0,1]$, $g_T^1\in \partial \psi(T_{k,i+1}^1)$, $g_T^2\in \partial \psi(T_{k,i+1}^2)$, and   $g_{k,i+1}=\left(\alpha_{k,i+1}\| \nabla f(T_{k,i+1}^1)+g_T^1\|_*^{\tfrac{p+1}{p}}+(1-\alpha_{k,i+1})\| \nabla f(T_{k,i+1}^2)+g_T^2\|_*^{\tfrac{p+1}{p}}\right)^{\tfrac{p}{p+1}}$;\;
        $i=i+1$;\;
     } 
     $i_k=i$;\;
}
\caption{ Bisection Scheme for Auxiliary Segment Search\label{alg:bisectionPrc}}
\end{algorithm}
\vspace{3mm}

We next provide a bound on the maximum number iterations of Algorithm \ref{alg:bisectionPrc} before the stopping criterion $\alpha_{k,i}(\tau_{k,i}^1-\tau_{k,i}^2)\beta_{k,i}^1\leq \tfrac{1}{2} \left(\tfrac{1-\beta}{H}\right)^{1/p} g_{k,i}^{(p+1)/p}$ is satisfied.

\begin{thm}[number of bisections in each step]\label{thm:compAlg}
    In Algorithm \ref{alg:bisectionPrc}, let the following hold 
    \begin{equation}\label{eq:minFtki1Ftki2}
        \min\set{F(T_{k,i}^1),F(T_{k,i}^2)}\geq F^*+\varepsilon, \quad i=0,\ldots,i_k, 
    \end{equation}
    for the accuracy parameter $\varepsilon>0$. Then, setting $\beta\in [0,3/(3p+2)]$ and denoting $(\tau)_+=\max\set{\tau,0}$ for $\tau\in\R$, we have
    \begin{equation}\label{eq:ik}
        i_k\leq \left(2+\tfrac{1}{p}\log_2\tfrac{5HD_*}{4(1-\beta)\varepsilon}\right)_+.
    \end{equation}
\end{thm}

\begin{proof}
    The bisection updates of Algorithm \ref{alg:bisectionPrc} leads to
    \begin{equation}\label{eq:tki1tki2}
        T_{k,i}^2-T_{k,i}^1=2^{-i}, \quad \beta_{k,i}^1 \leq 0\leq \beta_{k,i}^2.
    \end{equation}
    Hence, the Cauchy-Schwartz inequality yields
    \begin{align*}
        -\alpha_{k,i}\beta_{k,i}^1&=\tfrac{-\beta_{k,i}^1\beta_{k,i}^2}{\beta_{k,i}^2-\beta_{k,i}^1} \leq \min\set{-\beta_{k,i}^1,\beta_{k,i}^2}\\
        &\leq \|u_k\| \min\set{\| \nabla f(T_{k,i}^1)+g_T^1\|_*,\| \nabla f(T_{k,i}^2)+g_T^2\|_*}\\
        &\leq D_0\min\set{\| \nabla f(T_{k,i}^1)+g_T^1\|_*,\| \nabla f(T_{k,i}^2)+g_T^2\|_*}\\
        &\leq 2D_*\min\set{\| \nabla f(T_{k,i}^1)+g_T^1\|_*,\| \nabla f(T_{k,i}^2)+g_T^2\|_*},
    \end{align*}
    for $g_T^1\in \partial\psi(T_{k,i}^1)$ and $g_T^2\in \partial\psi(T_{k,i}^2)$. Note that by \eqref{eq:normxkx*D*vkx*D*} and $\tau\in [0,1]$, we get 
    \begin{align*}
        \|x_k+\tau u_k-x^*\|=\|(1-\tau)(x_k-x^*)+\tau(\upsilon-x^*)\|\leq (1-\tau)\|x_k-x^*\|+\tau\|\upsilon-x^*\|\leq D_*.
    \end{align*}
    Since $\beta\in [0,3/(3p+2)]$ for $p\geq 2$ (i.e., $0\leq \beta\leq \tfrac{3}{8}$), it can be deduced by Lemma \ref{lem:AnormIneq} that 
    \begin{align*}
        \|T_{k,i}^1-x^*\|\leq \tfrac{5}{4}D_*, \quad \|T_{k,i}^2-x^*\|\leq \tfrac{5}{4}D_*.
    \end{align*}
    Combining with the definition of $g_{k,i}$, the subgradient inequality, $-\alpha_{k,i}\beta_{k,i}^1\geq 0$, and \eqref{eq:minFtki1Ftki2}, this ensures
    \begin{align*}
        \tfrac{\tfrac{1}{2} \left(\tfrac{1-\beta}{H}\right)^{1/p} g_{k,i}^{(p+1)/p}}{-\alpha_{k,i}\beta_{k,i}^1}
        &\geq \tfrac{\tfrac{1}{2} \left(\tfrac{1-\beta}{H}\right)^{1/p} \min\set{\| \nabla f(T_{k,i}^1)+g_T^1\|_*,\| \nabla f(T_{k,i}^2)+g_T^2\|_*}^{(p+1)/p}}{2D_*\min\set{\| \nabla f(T_{k,i}^1)+g_T^1\|_*,\| \nabla f(T_{k,i}^2)+g_T^2\|_*}}\\
        &\geq \eta \min\set{\innprod{\nabla f(T_{k,i}^1)+g_T^1}{T_{k,i}^1-x^*},\innprod{\nabla f(T_{k,i}^2)+g_T^2}{T_{k,i}^2-x^*}}^{1/p}\\
        &\geq \eta \min\set{F(T_{k,i}^1)-F(x^*),F(T_{k,i}^2)-F(x^*)}^{1/p}\\
        &\geq \eta \varepsilon^{1/p},
    \end{align*}
    where $\eta= \tfrac{1}{4D_*} \left(\tfrac{1-\beta}{H}\right)^{1/p} \left(\tfrac{4}{5D_*}\right)^{1/p}$. Setting $2^{-i_k} \leq \tfrac{1}{4D_*} \left(\tfrac{1-\beta}{H}\right)^{1/p} \left(\tfrac{4}{5D_*}\right)^{1/p} \varepsilon^{1/p}$ leads to our desired result.
\end{proof}

\section{Superfast high-order methods with segment search}\label{sec:supFastSS}
This section concerns with the development of high-order methods for solving the composite minimization problem \eqref{eq:prob} in sense of the BiOPT framework given in \cite{ahookhosh2020inexact}. We recall that this framework involves two levels of methodologies. On one hand, at the upper level, we consider the accelerated inexact proximal-point method with a segment search given in Algorithm~\ref{alg:iaihoppaSS} for arbitrary $p$.  On the other hand, corresponding auxiliary problems are solved by a non-Euclidean composite gradient method (see Algorithm~\ref{alg:zk1}), while the segment search is handled by Algorithm \ref{alg:bisectionPrc}. Setting $q=\floor{p/2}$ and $H=\tfrac{6}{(p-1)!} M_{p+1}(f)$, we will show that this algorithm the complexity
\begin{equation}
    \mathcal{O}\left(\left[\tfrac{HD_*^{p+1}}{\varepsilon}\right]^{\tfrac{2}{6q+1}}\ln \tfrac{HD_*^{p+1}}{\varepsilon}\right),\quad p\geq 2
\end{equation} 
if $p$ is even, which is close to the optimal bound $\mathcal{O}(\varepsilon^{-2/(6q+1)})$. Moreover, it attains the superfast complexity
\begin{equation}
    \mathcal{O}\left(\left[\tfrac{HD_*^{p+1}}{\varepsilon}\right]^{\tfrac{1}{3q+1}}\ln \tfrac{HD_*^{p+1}}{\varepsilon}\right),\quad p\geq 2
\end{equation} 
if $p$ is odd, which is better than the optimal bound $\mathcal{O}(\varepsilon^{-2/(6q+1)})$ for the accuracy parameter $\varepsilon>0$.

Let us begin by describing the non-Euclidean composite gradient method for finding an acceptable solution $z_k$ satisfying \eqref{eq:A}, which was introduced and analyzed in \cite{ahookhosh2020inexact}. We here summarize the algorithm and results of \cite{ahookhosh2020inexact} for the sake of self-containedness. Let $\func{\rho}{\E}{\R}$ be a closed, convex, and differentiable function. Then, the \textit{Bregman} distance is a non-symmetric distance function $\func{\beta_\rho}{\E\times\E}{\R}$ given by
\begin{equation}\label{eq:BregmanDist}
    \beta_\rho(x,y)=\rho(y)-\rho(x)-\innprod{\nabla \rho(x)}{y-x}.
\end{equation}
For a convex function $\func{h}{\E}{\R}$, we then say that $h(\cdot)$ is \textit{$L_h$-smooth relative to $\rho(\cdot)$} if there exists a constant $L_h>0$ such that $L_h\rho-h$ is convex, and we call it \textit{$\mu_h$-strongly convex relative to $\rho(\cdot)$} if there exists $\mu_m>0$ such that $h-\mu_h\rho$ is convex; see \cite{bauschke2016descent,lu2018relatively}. The constant $\kappa_h=\mu_h/L_h$ is called the \textit{condition number} of $h(\cdot)$ relative to $\rho(\cdot)$. 


In order to provide an {\it acceptable solution} for the auxiliary problem \eqref{eq:A}, we need to minimize the function $\func{\varphi_k}{\E}{\R}$ given by
\begin{equation}\label{eq:psik}
    \varphi_k(z)=f_{y_k,H}(z)+\psi(z), \quad \forall k\geq 0,~ z\in\dom \psi,
\end{equation}
where the solution should fulfill the inequality in \eqref{eq:A}. To do so, we first consider the function $\func{\rho_{y_k,H}}{\E}{\R}$ given by
\begin{equation}\label{eq:rhoBarxPth}
    \rho_{y_k,H}(x)= \sum_{k=1}^{q} \tfrac{1}{(2k)!} D^{2k} f(y_k)[x-y_k]^{2k}+H d_{p+1}(x-y_k),
\end{equation}
which we denote by $\rho_k(\cdot)$ for sake of simplicity. It is notable from \cite[Theorems~3.5 and 3.6]{ahookhosh2020inexact} that
\begin{description}
    \item[{\bf(i)}] $\rho_k(\cdot)$ is a uniformly convex scaling function of the degree $p+1$ with the modulus $2^{2-p}$, which is $\beta_{\rho_k}(x,y)\geq \tfrac{2^{2-p}}{(p+1)} \|y-x\|^{p+1}$; see \cite[Theorem~3.5]{ahookhosh2020inexact};
    \item[{\bf(ii)}] the function $f_{y_k,H}^p(\cdot)$ is $L$-smooth and $\mu$-strongly convex relative to the scaling function $\rho_k(\cdot)$ for constants 
    \begin{align*}
    \mu=1-\tfrac{1}{\xi}, \quad L=1+\tfrac{1}{\xi}, \quad \kappa = \tfrac{\xi-1}{\xi+1},
\end{align*}
where $\xi$ is the unique solution of the quadratic equation $\xi(1+\xi) = \tfrac{(p-1)!H}{M_{p+1}(f)}$; see \cite[Theorem~3.6]{ahookhosh2020inexact}. In light of this equation, in the remainder of this section, we consider the following parameters
\begin{equation}\label{eq:parSuperfastAlg1}
    \xi=2,\quad H=\tfrac{6}{(p-1)!} M_{p+1}(f), \quad \mu=\tfrac{1}{2}, \quad L=\tfrac{3}{2}, \quad \kappa=\tfrac{1}{3};
\end{equation}

    \item[{\bf(iii)}] $\|\nabla^2 \rho_{y_k,H}(\cdot)\|$ is bounded on the bounded set
\begin{align*}\label{eq:levelSetPsik}
    \mathcal{L}_k(z_0, \Delta) = \set{z\in\E ~:~ \|z_0-z\|\leq \Delta,~\varphi_k(z)\leq \varphi_k(z_0)}
\end{align*}
for some $\Delta>0$ and if $M_{2}(f)<+\infty$, $M_{4}(f)<+\infty$, and $M_{p+1}(f)<+\infty$; cf. \cite[Theorem~3.5]{ahookhosh2020inexact}.
\end{description}

Let us assume that points $y_k,z_i\in\E$ and constant $H>0$ are given. In order to minimize \eqref{eq:psik} inexactly in the sense \eqref{eq:A}, we consider the non-Euclidean proximal scheme
\begin{equation}\label{eq:zk1}
    z_{i+1} = \argmin_{z\in\E}\set{\innprod{\nabla f_{y_k,H}(z_i)}{z-z_i}+\psi(z)+2L_f \br(z_i,z)},
\end{equation}
where $z_k^*$ denotes its optimal solution. It follows from \eqref{eq:rhoBarxPth} that
\begin{align*}
        \beta_{\rho_k}(z_i,z) &= \rho_k(z)-\rho_k(z_i)-\innprod{\nabla \rho_k(z_i)}{z-z_i}\\
        &= \sum_{k=1}^{q} \tfrac{1}{(2k)!} D^{2k} f(y_k)[z-y_k]^{2k}-\sum_{k=1}^{q} \tfrac{1}{(2k)!} D^{2k} f(y_k)[z_i-y_k]^{2k}\\
        &~~~-\innprod{\sum_{k=1}^{q} \tfrac{1}{(2k-1)!} D^{2k} f(y_k)[z_i-y_k]^{2k-1}}{z-z_i}\\
        &~~~+H \left(d_{p+1}(z-y_k)- d_{p+1}(z_i-y_k)-\innprod{\nabla d_{p+1}(z_i-y_k)}{z-z_i}\right).
    \end{align*}
	Hence, we need to solve the auxiliary problem  
\begin{equation*}
	\begin{split}
    z_{i+1} &= \argmin_{z\in\E}\set{\innprod{\nabla f_{y_k,H}(z_i)}{z-z_i}+\psi(z)+2L \br(z_i,z)}\\
    & = \argmin_{z\in\E}\left\{\innprod{\nabla f_{y_k,H}(z_i)-2L\sum_{k=1}^{q} \tfrac{1}{(2k-1)!} D^{2k} f(y_k)[z_i-y_k]^{2k-1}}{z-z_i}+\psi(z) \right.\\
    &\left. ~~~~~~~~~~~~~~~~~~+2L\sum_{k=1}^{q} \tfrac{1}{(2k)!} D^{2k} f(y_k)[z-y_k]^{2k}+2LH \beta_{d_{p+1}}(z_i,z)\right\}
    \end{split}
\end{equation*}
This leads to the minimization problem
\begin{align}\label{eq:zk0}
    \min_{z\in\dom\psi} \set{\innprod{c}{z}+2L\sum_{k=1}^{q} \tfrac{1}{(2k)!} D^{2k} f(y_k)[z]^{2k}+\psi(z)+\tfrac{2LH}{p+1}\|z\|^{p+1}},
\end{align}
where
\begin{align*}
    c=\nabla f_{y_k,H}(z_i)-2L\sum_{k=1}^{q} \tfrac{1}{(2k-1)!} D^{2k} f(y_k)[z_i-y_k]^{2k-1}-2LH \|z_i\|^{p-1} B(z_i-y_k).
\end{align*}
For the given sequence $\{z_i\}_{i\geq 0}$ given by \eqref{eq:zk0}, we will stop the scheme as soon as the inequality 
\begin{align*}
    \|\nabla f_{y_k, H}^p(z_{i+1})+g\|_*\leq \beta \|\nabla f(z_{i+1})+g\|_*
\end{align*}
holds, and then we set $z_k=z_{i+1}$. In the remainder of this section, we show that this inequality holds for large enough internal iterations $i$. 

Summarizing above-mentioned discussion, we come to the following non-Euclidean composite gradient algorithm.

\vspace{2mm}
\RestyleAlgo{boxruled}
\begin{algorithm}[H]
\DontPrintSemicolon \KwIn{$z_{0}=y_k\in \dom \psi$,~ $\beta\in [0,3/(3p+2)]$,~ $L>0$,~ $i=0$;} 
\Begin{ 
    \Repeat{$\|\nabla f_{y_k, H}^p(z_i)+g\|_*\leq \beta \|\nabla f(z_i)+g\|_*$}{  
        Compute $z_{i+1}$ by \eqref{eq:zk0};\;
        Set $g=L(\nabla \rho(z_i)-\nabla \rho(z_{i+1}))-\nabla f_{y_k,H}^p(z_i)\in\partial\psi(z_{i+1})$ and $i=i+1$;
     } 
     $i_k^* = i$;\;
}
\KwOut{$z_k=z_{i_k^*}$ and $g=L(\nabla \rho(z_{i_k^*-1})-\nabla \rho(z_{i_k^*}))-\nabla f_{y_k,H}^p(z_{i_k^*-1})\in\partial\psi(z_{i+1})$.}
\caption{ Non-Euclidean Composite Gradient Algorithm\label{alg:zk1}}
\end{algorithm}
\vspace{2mm}

Let us set $S=\set{x\in\dom \psi ~:~ \|z-x^*\| \leq 2 R_0}$ and assume
\begin{equation*}
    \psi_S = \sup_{z\in S} F(z)<+\infty.
\end{equation*}
Moreover, for the accuracy parameter $\varepsilon>0$, we assume
\begin{equation*}\label{eq:assumPsiPsi*}
    F(z_i)-F(x^*)\geq \varepsilon, \quad \forall  i_k^*\geq i\geq 0,~ \forall k\geq 0,
\end{equation*}
and show that Algorithm \ref{alg:zk1} is well defined. 

\begin{thm}[well-definedness of Algorithm \ref{alg:zk1}]\cite[Theorem~3.5 and Corollary~3.6]{ahookhosh2020inexact}
\label{thm:wellDefinedAlg3}
Let Algorithm~\ref{alg:iaihoppaSS} be applied to the problem \eqref{eq:prob}, where its auxiliary proximal-point problem \eqref{eq:prox00} is minimized inexactly by Algorithm~\ref{alg:zk1}.
Let $\beta\in [0,3/(3p+2)]$, $z_0=y_k$, and $\{z_i\}_{i\geq 0}$ be a sequence generated by Algorithm~\ref{alg:zk1}, and let
    \begin{equation}\label{eq:Fzix}
        F(z_i)-F(x^*)\geq \varepsilon, \quad \forall  i\geq 0,
    \end{equation}
    where $x^*$ is a minimizer of $F$ and $\varepsilon>0$ is the accuracy parameter. Moreover, assume that there exists a constant $D>0$ such that $\|z_i-x^*\|\leq D$ for all $i\geq 0$.
    Then, for the subgradients 
    \begin{align*}
        \mathcal{G}_{i_k^*} = \nabla f_{y_k,H}^p(z_{i_k^*})+g\in \partial \varphi_k (z_{i_k^*}),\quad g=L(\nabla \rho(z_{i_k^*-1})-\nabla \rho(z_{i_k^*}))-\nabla f_{y_k,H}^p(z_{i_k^*-1}) \in\partial\psi(z_{i_k^*}),
    \end{align*}
    and $z_{i_k^*}\in\dom \psi$, the maximum number of iterations $i_k^*$ needed to guarantee the inequality
    \begin{equation}\label{eq:AIneqNEPA2}
        \|\mathcal{G}_{i_k^*}\|_* \leq \beta \|\nabla f(z_{i_k^*})+g\|_*
    \end{equation}
    satisfies 
    \begin{equation}\label{eq:lowerBoundi1}
        i_k^*\leq 1+ \tfrac{2(p+1)}{\kappa}\log\left(\tfrac{\tfrac{D}{\beta} \left(\tfrac{2L}{C}\beta_{\rho}(z_0,z_k^*)\right)^{1/(p+1)}}{ \varepsilon} \right),
    \end{equation}
    where $C=\tfrac{L2^{2-p}}{(p+1)(L-\mu)^{p+1} \overline L^{p+1}}$ and $\varepsilon>0$ is the accuracy parameter.
\end{thm}

We note that the auxiliary problem \eqref{eq:prox00} can be solved inexactly by applying one iteration of the tensor method (e.g., \cite{nesterov2019implementable}) such that the inequality given in \eqref{eq:A} is satisfied. For detailed discussion, we refer the readers to Section~2.1 of \cite{ahookhosh2020inexact}.

We next present our superfast high-order method by combining all above facts with Algorithm~\ref{alg:iaihoppaSS} leading to the following algorithm.

\RestyleAlgo{boxruled}
\begin{algorithm}[ht!]
\DontPrintSemicolon \KwIn{$x_{0}\in\dom\psi$,~ $\upsilon_0=x_0$,~ $\beta\in [0,3/(3p+2)]$,~ $H=\tfrac{6}{(p-1)!}M_{p+1}(f)$,~$A_0=0$,~ $\Psi_0=\tfrac{1}{2}\|x-x_0\|^2$,~ $k=0$;} 
\Begin{ 
    \While {stopping criterion does not hold}{ 
        Set $u_k=\upsilon_k-x_k$ and $z_0=y_k=x_k$;\;  
        Find an acceptable solution $x_k^0=z_{i_k^*}$ of \eqref{eq:prox00} and $g\in\partial \psi(x_k^0)$ by Algorithm \ref{alg:zk1} such that $(x_k^0,g)\in\A(y_0,\beta)$;\;
        \uIf{$\innprod{\nabla f(x_k^0)+g}{u_k}\geq 0$}{
            $\eL_k(x)=\ell_{x_k^0}(x)+\psi(x)$,
            $x_{k+1}=x_k^0,~ g_k=\| \nabla f(x_k^0)+g\|_*$;\;
        }
        \Else{
            Set $z_0=y_k=\upsilon_k$ and find an acceptable solution $x_k^1=z_{i_k^*}$ of \eqref{eq:prox00} and $\overline{g}\in\partial\psi(x_k^1)$ by Algorithm \ref{alg:zk1} such that $(x_k^1,\overline{g})\in\A(y_k,\beta)$;\;
            \uIf{$\innprod{\nabla f(x_k^1)+\overline{g} }{u_k}\leq 0$}{
                $\eL_k(x)=\ell_{x_k^1}(x)+\psi(x)$,
                $x_{k+1}=x_k^1,~ g_k=\| \nabla f(x_k^1)+\overline{g} \|_*$ for $\overline{g} \in\partial \psi(x_k^1)$;\;
            }
            \Else{
                Apply Algorithm \ref{alg:bisectionPrc} to find $0\leq\tau_k^1\leq\tau_k^2\leq 1$, $y_k^1=x_k+\tau_k^1u_k$, $(T_k^1,\widehat{g})\in\A(y_k^1,\beta)$ for $\widehat{g}\in\partial \psi(T_k^1)$, $y_k^2=x_k+\tau_k^2u_k$, $(T_k^2,\widetilde{g})\in\A(y_k^2,\beta)$ for $\widetilde{g}\in\partial \psi(T_k^2)$ such that
                \begin{align*}\label{eq:else4}
                    \beta_k^1\leq 0 \leq \beta_k^2,~ \alpha_k(\tau_k^1-\tau_k^2)\beta_k^1\leq \tfrac{1}{2} \left(\tfrac{1-\beta}{H}\right)^{1/p} g_k^{(p+1)/p},~ \alpha_k = \tfrac{\beta_k^2}{\beta_k^2-\beta_k^1}\in [0,1]
                \end{align*}
                with $\beta_k^1=\innprod{\nabla f(T_k^1)+\widehat{g}}{u_k}$, $\beta_k^2=\innprod{\nabla f(T_k^2)+\widetilde{g}}{u_k}$, 
                \begin{align*}
                    g_k=\left(\alpha_k\| \nabla f(T_k^1)+\widehat{g}\|_*^{(p+1)/p}+(1-\alpha_k)\| \nabla f(T_k^2)+\widetilde{g}\|_*^{(p+1)/p}\right)^{p/(p+1)};
                \end{align*}
                Set $\eL_k(x)=\alpha_k \ell_{T_k^1}(x) +(1-\alpha_k) \ell_{T_k^2}(x)+\psi(x)$ and $x_{k+1}=\alpha_k T_k^1+(1-\alpha_k)T_k^2$;\; 
            }
        }
        Compute $a_{k+1}$ by solving  $\tfrac{a_{k+1}^2}{A_{k+1}+a_{k+1}}=\tfrac{1}{2}\left(\tfrac{1-\beta}{H}\right)^{1/p}g_k^{(1-p)/p}$, and $A_{k+1}=A_k+a_{k+1}$;\;
        Set $\Psi_{k+1}(x)=\Psi_k(x)+a_{k+1}\eL_k(x)$ and compute $\upsilon_{k+1}=\argmin_{x\in\dom\psi} \Psi_{k+1}(x)$;\;
     } 
}
\caption{Superfast High-Order Segment Search Algorithm\label{alg:superFSOA}}
\end{algorithm}
\vspace{3mm}

Let us define the {\it norm-dominated scaling function} to upper bound the Bregman term $\beta_{\rho_k}(\cdot,\cdot)$.

\begin{defin}\cite[Definition 2]{nesterov2020superfast}
\label{def:normDominat}
    The scaling function $\rho(\cdot)$ is called \emph{norm-dominated} on the set $S\subseteq \E$ by some function $\func{\theta_S}{\R_+}{\R_+}$ if $\theta_S(\cdot)$ is convex with $\theta_S(0)=0$ such that
    \begin{equation}\label{eq:normDominat}
        \br(x,y) \leq \theta_S(\|x-y\|),
    \end{equation}
    for all $x\in S$ and $y\in\E$.
\end{defin}

\begin{lem}[norm-dominatedness of the scaling function $\rho_k(\cdot)$]\cite[Lemma~3.9 and Lemma~3.10]{ahookhosh2020inexact}\label{lem:rhokGradDomin}
    Let $p\geq 2$ and $q=\floor{p/2}$. Then, the scaling function $\rho_k(\cdot)$ is norm-dominated on the Euclidean ball $B(y_k,D_1)=\set{x\in\E\mid \|x-y_k\|\leq D_1}$ by
    \begin{equation}\label{eq:rhokGradDomin}
        \theta_B(\tau) = \tfrac{\widehat{L}}{2} \tau^2+\left\{
        \begin{array}{ll}
           \alpha_1 a_1 \tau^{2q+2}+ \beta_1 d_1  &~~ \mathrm{if}~ p=2q+1,\vspace{2mm}\\
           \alpha_2 a_2 \tau^{2q+1}+ \beta_2 d_2  &~~ \mathrm{if}~ p=2q,
        \end{array}
        \right.
    \end{equation}
     where $\widehat{L}=\sum_{k=1}^{q} \tfrac{4(1-(2D_1^2)^{2k-1})\sum_{i=1}^{2k-1}\binom{2k-1}{i}}{(1-2D_1^2)(2k-1)!}$ and
    \begin{equation*}
        \begin{array}{l}
        1<\alpha_1 \leq 1+\tfrac{1}{2a_1}(b_1+1)c_1^{-\tfrac{q+1}{q}}, \quad 1+\tfrac{1}{2d_1} (b_1+1) c_1^{\tfrac{q+1}{q}}\leq \beta_1,\\
        1<\alpha_2 \leq 1+\tfrac{b_2+1}{2a_2}c_2^{-\tfrac{2q+1}{2(2q-1)}}, \quad 1+\tfrac{b_2+1}{2d_2} c_2^{\tfrac{2q+1}{2(2q-1)}}\leq \beta_2,
        \end{array}
    \end{equation*}
    with 
    \begin{equation*}
        \begin{array}{l}
         a_1=\tfrac{2^{2q}}{p+1}, \quad b_1=\tfrac{2^{3q+1}}{p+1} R^{q+1},\quad c_1=R^{p}
, \quad d_1=\tfrac{2^q}{p+1} R^{2q+2},\\
         a_2=2^{2q-1}, \quad b_2=\tfrac{2^{\tfrac{6q-1}{2}}}{p+1} R^{\tfrac{2q+1}{2}},\quad c_2=R^{p}, \quad d_2=\tfrac{2^{\tfrac{2q-1}{2}}}{p+1} R^{2q+1},
        \end{array}
    \end{equation*}
    for $\tau\geq 0$.
\end{lem}

We conclude this section by providing the complexities of Algorithm \ref{alg:superFSOA} in both upper and lower levels, which is the main result of this section.

\begin{thm}[complexity of Algorithm \ref{alg:superFSOA}]
\label{thm:CompSupFP1OA}
    Let us assume that all conditions of Theorem~\ref{thm:wellDefinedAlg3} hold, and let $p\geq 2$, $q=\floor{p/2}$, and $\beta\in [0,3/(3p+2)]$. Then, Algorithm \ref{alg:superFSOA} attains an $\varepsilon$-solution of the problem \eqref{eq:prob} in
    \begin{equation}\label{eq:complexitySFSOA}
       \tfrac{1-p}{2}+\tfrac{p+1}{2} \left(\tfrac{6 (4^p) (3p+2) M_{p+1}(f)R_0^{p+1}}{(3p-1)(p-1)! \varepsilon}\right)^{\tfrac{2}{3p+1}}
    \end{equation}
    iterations, for the accuracy parameter $\varepsilon>0$. Moreover, the auxiliary problem \eqref{eq:zk0} is approximately solved by Algorithm \ref{alg:zk1} in at most
    \begin{equation}\label{eq:ikStarUp}
        i_k^*\leq 1+ \tfrac{2(p+1)}{\kappa}\log\left(\tfrac{\tfrac{D(L-\mu)\overline{L}}{\beta} \left((p+1)2^p \theta_B(D_1)\right)^{1/(p+1)}}{ \varepsilon} \right),
    \end{equation}
    iterations. 
\end{thm}

\begin{proof}
    From $H=\tfrac{6}{(p-1)!}M_{p+1}(f)$ and Theorem \ref{thm:convRateAlg6}, we obtain
    \begin{align*}
        F(x_k)-F^* \leq \tfrac{4^p H R_0^{p+1}}{1-\beta} \left(1+\tfrac{2(k-1)}{p+1}\right)^{-\tfrac{3p+1}{2}}
        \leq \tfrac{6(4^p)(3p+2) M_{p+1}(f) R_0^{p+1}}{(3p-1)(p-1)!} \left(1+\tfrac{2(k-1)}{p+1}\right)^{-\tfrac{3p+1}{2}},
    \end{align*}
    which consequently leads to the complexity \eqref{eq:complexitySFSOA} of Algorithm~\ref{alg:superFSOA}.
    
    For all $x\in B(y_k,D_1)$, the definition of $\theta_B(\cdot)$ in \eqref{eq:rhokGradDomin} implies that $\theta_B(\|z_k^*-y_k\|) \leq \theta_B(D_1)$.
    Accordingly, it follows from $\beta\in [0,3/(3p+2)]$, $C=\tfrac{L2^{1-p}}{(p+1)(L-\mu)^{p+1} \overline L^{p+1}}$, and \eqref{eq:lowerBoundi1} that
    \begin{align*}
        i_k^*&\leq 1+ \tfrac{2(p+1)}{\kappa}\log\left(\tfrac{\tfrac{D}{\beta} \left(\tfrac{2L}{C}\beta_{\rho}(z_0,z_k^*)\right)^{1/(p+1)}}{ \varepsilon} \right)\\
        &\leq 1+ \tfrac{2(p+1)}{\kappa}\log\left(\tfrac{\tfrac{D}{\beta} \left(\tfrac{2L}{C}\theta_B(\|y_k-z_k^*\|)\right)^{1/(p+1)}}{ \varepsilon} \right)\\
        &\leq 1+ \tfrac{2(p+1)}{\kappa}\log\left(\tfrac{\tfrac{D}{\beta} \left(\tfrac{2L}{C}\theta_B(D_1)\right)^{1/(p+1)}}{ \varepsilon} \right),
    \end{align*}
    adjusting the inequality \eqref{eq:ikStarUp}. 
\end{proof}

Note that if $q\geq 1$, the scaling function $\rho_k(\cdot)$ \eqref{eq:rhoBarxPth} is the same for both $p=2q$ and $p=2q+1$, which requires $2q$th-order oracles. As such, if $p$ is even (i.e., $p=2q$), then Algorithm~\ref{alg:superFSOA} is a $2q$th-order method and obtaining the complexity of order $\mathcal{O}(\varepsilon^{-2/(6q+1)})$, which the same as the optimal complexity bound. However, if $p$ is odd ($p=2q+1$), then Algorithm~\ref{alg:superFSOA} is again a $2q$th-order method attaining the complexity of order $\mathcal{O}(\varepsilon^{-1/(3q+1)})$, which is always better than the optimal complexity $\mathcal{O}(\varepsilon^{-2/(6q+1)})$; cf. \cite{nesterov2020inexact}.

\subsection{Application to a structured separable function}\label{sec:appSSF}
    Let us consider the univariate functions $\func{f_i}{\R}{\R}$ that are $p$ times continuously differentiable, for $i=1,\ldots,N$. Then, we define the function $\func{f}{\R^n}{\R}$ as
    \begin{equation}\label{eq:funcExa}
        f(x)=\sum_{i=1}^N f_i(\innprod{a_i}{x}-b_i),
    \end{equation}
    where $b\in\R^N$ and $a_i\in\R^n$. Here, we are interested to the composite minimization of the form \eqref{eq:prob}, i.e.,
    \begin{equation}\label{eq:strSepFunc}
        \min_{x\in\dom \psi}~\set{\sum_{i=1}^N f_i(\innprod{a_i}{x}-b_i)+\psi(x)}.
    \end{equation}
    Therefore, for $q=\floor{p/2}$, Algorithm~\ref{alg:superFSOA} is a $2q$th-order method with the convergence rate $\mathcal{O}(k^{-\tfrac{3p+1}{2}})$. In the remainder of this section, we verify efficient implementation of Algorithm~\ref{alg:superFSOA} to this problem. 
    
    \paragraph{{\bf Second-order method with complexity $\mathcal{O}(\varepsilon^{-(2/(3p+1))})$.}} We now consider applying Algorithm~\ref{alg:superFSOA} the problem \eqref{eq:strSepFunc} for an arbitrary $p$. Setting $q=\floor{p/2}$ and considering the scaling function \eqref{eq:rhoBarxPth}, it is required to solve the subproblem 
    \begin{align*}
        \min_{z\in\dom\psi} \set{\innprod{c}{z}+2L\sum_{k=1}^{q} \tfrac{1}{(2k)!} D^{2k} f(y_k)[z]^{2k}+\psi(z)+\tfrac{2LH}{p+1}\|z\|^{p+1}},
    \end{align*}
    with
    \begin{align*}
        c=\nabla f_{y_k,H}(z_i)-2L\sum_{k=1}^{q} \tfrac{1}{(2k-1)!} D^{2k} f(y_k)[z_i-y_k]^{2k-1}-2LH \|z_i\|^{p-1} B(z_i-y_k).
    \end{align*}
    Thus, for implementation of Algorithm~\ref{alg:superFSOA}, we need $2q$th-order oracle of $f_i(\cdot)$, for $i=1,\ldots,N$. Concidering the structure of the function $f(\cdot)$, it is clear that 
    \begin{align*}
        &\innprod{\nabla^2 f(y_k)z}{z}= \sum_{i=1}^N \nabla^2 f_i(\innprod{a_i}{y_k}-b_i) \innprod{a_i}{z}^2,\\ 
        &D^{2j} f(y_k)[z]^{2j}= \sum_{i=1}^N \nabla^{2j} f_i(\innprod{a_i}{y_k}-b_i) \innprod{a_i}{z}^{2j} \quad 1\leq j\leq q.
    \end{align*}
    Let us particularly verify these terms for $f_i(x)=-\log(x)$ ($i=1,\ldots,N$) for $x\in (0,+\infty)$, which consequently leads to
    \begin{align*}
        \nabla^2f_i(x)=\tfrac{1}{x^2}, \quad \nabla^{2j}f_i(x)=\tfrac{(2j)!}{x^4}=(2j)!\left(\nabla^2f_i(x)\right)^j\quad 1\leq j\leq q,
    \end{align*}
    i.e.,
    \begin{equation}\label{eq:d4Example}
        D^{2j} f(y_k)[z]^{2j} = (2j)!\sum_{i=1}^N \left(\nabla^2f_i(\innprod{a_i}{y_k}-b_i)\right)^j \innprod{a_i}{z}^{2j}.
    \end{equation}
    As a result, the implementation of Algorithm~\ref{alg:superFSOA} with even $p=2q$ and odd $p=2q+1$ only requires the second-order oracle of $f_i(\cdot)$ ($i=1,\ldots,N$) and the first-order oracle of $\psi(\cdot)$. Hence, we end up with a second-order method with the complexity of order $\mathcal{O}(\varepsilon^{-2/(6q+1)})$ for $p=2q$ and $\mathcal{O}(\varepsilon^{-1/(3q+1)})$ for $p=2q+1$ (see Theorem~\ref{thm:CompSupFP1OA}), which are much faster than the second-order methods optimal bound $\mathcal{O}(\varepsilon^{-2/7})$ for $p\geq 2$. Finally, these complexities are better than those of 
    \cite[Algorithm~4]{ahookhosh2020inexact}.

\section{Conclusion}\label{sec:conclusion}
In this paper, we investigated a general framework (BiOPT) for designing and developing high-order methods for convex composite functions attaining the complexity better than the classical optimal  bounds. In this framework, we first regularized the objective function with by a power of the Euclidean norm $\|\cdot\|^{p+1}$ for some $p\geq 1$ including a segment search along a specific direction. Minimizing this function leads to a high-order proximal-point segment search operator, which can be solves either exactly or inexactly. The BiOPT framework involves two levels of methodologies, upper and lower ones. At the upper level, an accelerated scheme was designed using approximate solution of the high-order proximal-point auxiliary minimization with segment search and the estimating sequence technique. At the lower level, the auxiliary problem was minimized inexactly by a combination of a non-Euclidean composite gradient method (approximating the solution of the high-order proximal-point minimization) and the bisection technique (approximating the segment search step-size). 

It was shown that if $p$ is even ($p=2q$), the proposed $2q$th-order method obtains the convergence rate $\mathcal{O}(k^{-(6q+1)/2})$ for $q=\floor{p/2}$, which is the same as the optimal rate of $2q$th-order methods. On the other hand, if $p$ is odd ($p=2q+1$), then the proposed $2q$th-order method attains the convergence rate $\mathcal{O}(k^{-(3q+1)})$, which is faster than the optimal convergence rate of $2q$th-order methods leading to a superfast method. We here emphasize that developing such methods with the complexity better than the classical optimal bounds is not a contradiction because our methods assumed the uniform boundedness of $(p+1)$th derivative ($M_{p+1}(f)>0$) which is stronger than the Lipschitz continuity of of $2q$th derivatives (see Table~\ref{tab:complexity}). 



\begin{table}[H]
    \centering
    \scalebox{0.85}{
    \begin{tabular}{|l|l|l|} 
    \hline
    Oracle & Analytical complexity & Assumptions\\ [0.5ex] 
    \hline
    $1$st & optimal, $\mathcal{O}(\varepsilon^{-\tfrac{1}{2}})$, e.g., \cite{nesterov2005smooth,nesterov2013gradient,neumaier2016osga,ahookhosh2019accelerated,ahookhosh2017optimal} & Lip. cont. of gradients \\
    \hline
    $2$nd & optimal, $\mathcal{O}(\varepsilon^{-\tfrac{2}{7}})$  & Lip. cont. of Hessians \\
    \hline
    $2$nd & superfast, $\mathcal{O}(\varepsilon^{-\tfrac{1}{4}})$ \cite{nesterov2020superfast,nesterov2020inexact}, $\mathcal{O}(\varepsilon^{-\tfrac{1}{5}})$ \cite{nesterov2020inexactSS} & $M_2(f), M_4(f)<\infty$ \\
    \hline
    $p$th & optimal, $\mathcal{O}(\varepsilon^{-\tfrac{2}{3p+1}})$ & Lip. cont. of $p$th derivatives \\
    \hline
    $2q$th & superfast, $p=3$, $\mathcal{O}(\varepsilon^{-\tfrac{1}{4}})$; suboptimal, $p\neq 3$, $\mathcal{O}(\varepsilon^{-\tfrac{1}{p+1}})$ \cite{ahookhosh2020inexact}  & $M_2(f), M_4(f), M_{p+1}(f)<\infty$ \\
    \hline
    $2q$th & optimal, $p=2q$, $\mathcal{O}(\varepsilon^{-\tfrac{2}{6q+1}})$; superfast, $p=2q+1$,  $\mathcal{O}(\varepsilon^{-\tfrac{1}{3q+1}})$ [*]  & $M_2(f), M_4(f), M_{p+1}(f)<\infty$ \\
    \hline
    \end{tabular}
    }
    \vspace{1mm}
    \caption{Analytical complexities of optimization methodologies with different assumptions for convex (composite) problems in which $q=\floor{p/2}$ and [*] stands for the current paper. \label{tab:complexity}}
\end{table}

An efficient implementation of algorithms based on BiOPT framework entails finding a solution for the auxiliary problem \eqref{eq:zk0} that consequently requires effective application of the non-Euclidean first-order method (Algorithm~\ref{alg:zk1}). We note that such implementations need first-order oracle of $f(\cdot)$ and $\psi(\cdot)$ and also the computation of $D^{2i} f(y_k)[h]^{2i-1}$, for $i=1,\ldots,N$. One may substantially improve this procedure by applying superlinearly convergent non-Euclidean methods (e.g., \cite{ahookhosh2021bregman}) for solving the auxiliary problem \eqref{eq:zk0}.

		\addcontentsline{toc}{section}{References}
		\bibliographystyle{spmpsci.bst}
	\bibliography{Bibliography}

\end{document}